\documentclass{amsart}
\usepackage{amsfonts, amsmath, amssymb, amscd, amsthm, graphicx,enumerate,float}
\usepackage{hyperref,xypic}
\usepackage[usenames, dvipsnames]{color}
\usepackage{oubraces}
\usepackage{todonotes}
\usepackage{overpic}
\usepackage{caption}
\usepackage{subcaption}

\usepackage{thmtools}
\usepackage{thm-restate}
\usepackage[normalem]{ulem}

\makeatletter
\def\blfootnote{\xdef\@thefnmark{}\@footnotetext}
\makeatother

\makeatletter 
\def\l@subsection{\@tocline{2}{0pt}{2pc}{6pc}{}} \makeatother

\newtheorem{thm}{Theorem}[section]
\newtheorem{cor}[thm]{Corollary}
\newtheorem{lem}[thm]{Lemma}
\newtheorem{prop}[thm]{Proposition}

\theoremstyle{definition}
\newtheorem{defn}[thm]{Definition}
\theoremstyle{remark}
\newtheorem{rem}[thm]{Remark}
\newtheorem{ex}[thm]{Example}

\newfont{\eufm}{eufm10}

\renewcommand{\phi}{\varphi}

\newcommand{\N}{\mathbb N}
\newcommand{\Z}{\mathbb Z}

\newcommand{\Map}{\operatorname{Map}}

\newcounter{ncomments}

\newcounter{pcomments}

\newcounter{ccomments}

	\makeatletter
\let\@wraptoccontribs\wraptoccontribs
\makeatother

\def\mc {\mathcal}

\begin{document}

\title{Shifts maps are not type-preserving}
\author{Carolyn Abbott}
\address{Department of Mathematics\\Brandeis University\\Waltham, MA 02453}
\email{carolynabbott@brandeis.edu}

\author{Nicholas Miller}
\address{Department of Mathematics\\University of Oklahoma\\Norman, OK 73019}
\email{nickmbmiller@ou.edu}

\author{Priyam Patel}
\address{Department of Mathematics\\University of Utah\\Salt Lake City, UT 84112}
\email{patelp@math.utah.edu}


\begin{abstract}
 For a surface $S$ of sufficient complexity, Dehn twists act elliptically on the arc, curve, and relative arc graph of $S$. We show that composing a Dehn twist with a shift map results in a loxodromic isometry of the relative arc graph $\mathcal{A}(S,p)$ for any surface $S$ with an isolated puncture $p$ admitting a shift map. Therefore, shift maps are not type-preserving.
\end{abstract}

\maketitle

\section{Introduction} 

A surface $S$ is \emph{finite-type} if its fundamental group is finitely generated, and is otherwise \emph{infinite-type}. The \emph{mapping class group}, $\Map(S)$, of a finite-type surface is well studied, especially through its actions on various hyperbolic graphs including the curve graph, $\mathcal{C}(S)$. The most simple mapping class, a Dehn twist about a simple closed curve, acts elliptically on $\mathcal{C}(S)$.

There have been many developments in the study of infinite-type surfaces and their mapping class groups over the last few years. For an infinite-type surface $S$ with at least one isolated puncture $p$, the \textit{relative arc graph}, $\mathcal{A}(S,p)$, plays the role of $\mathcal{C}(S)$ and is defined as follows: the vertices correspond to isotopy classes of simple arcs that begin and end at $p$, and edges connect vertices for arcs admitting disjoint representatives. The subgroup $\Map(S, p)$ of $\Map(S)$ that fixes the isolated puncture $p$ acts on $\mathcal{A}(S, p)$ by isometries. A Dehn twist about a simple closed curve acts elliptically on $\mathcal{A}(S, p)$ as well. 

This paper fits into a body of work aimed at constructing and classifying all of the elements of $\Map(S,p)$ acting loxodromically on $\mathcal{A}(S,p)$ and various other hyperbolic graphs associated to infinite-type surfaces (see \cite{Bavard, BavardWalker, AMP, PatelTaylor, MoralesValdez}). Our main result shows that \emph{shift maps} are not type-preserving in the sense that composing a Dehn twist with a shift map  results in a mapping class thats acts loxodromically on $\mathcal{A}(S,p)$.

\begin{thm}\label{thm:main}
Let $\chi$ be a standard simple closed curve in the biinfinite flute surface $S$ containing $\ell$ punctures besides $p$ in its interior. Then $g = h T_\chi^j$ is a  loxodromic isometry of $\mc A(S,p)$ for all $j>0$, unless both $\ell$ and $j$ are equal to 1, where $h$ is the standard shift on $S$. 
\end{thm}

\noindent For simplicity, we prove the theorem for the biinfinite flute surface $S$ and \emph{standard} curves (see Definition~\ref{def:standard}), but the result immediately extends to any surface $\Sigma$ containing an isolated puncture that admits a shift map since the inclusion of $\mc{A}(S, p)$ into $\mc{A}(\Sigma, p)$ is a $(2,0)$--quasi-isometric embedding. There are uncountably many such surfaces $\Sigma$, which are referred to as surfaces of type $\mathcal{S}$ (see \cite[Definition 2.6, Lemma 2.7, and Lemma 2.10]{AMP} for more details). In addition, Lemma~\ref{lem:standardscc} shows that we can extend Theorem~\ref{thm:main} to other simple closed curves as well. 

The proof that these mapping classes are loxodromic isometries of $\mc A(S,p)$ utilizes a ``starts like'' function, which (roughly) measures how long any arc starting at $p$ fellow travels arcs in a given collection. The function is then used to bound distance in $\mathcal{A}(S, p)$ from below and produce a quasi-axis for $g$.
 This method is inspired by Bavard's construction in \cite{Bavard} and that of the authors in \cite{AMP}.\\

\noindent\textbf{Acknowledgements:} Abbott was partially supported by NSF Grant DMS-2106906.  Miller was partially supported by NSF Grant DMS-2005438/2300370. Patel was partially supported by NSF CAREER Grant DMS–2046889 and the University of Utah Faculty Fellows Award.

\section{Background}
\subsection{Coding arcs and standard position}\label{sec:codingarcs}

Let  $S$ be the biinfinite flute surface with a distinguished isolated puncture $p$, and let $\{p_i\}_{i\in\Z}$ be the countably infinite discrete collection of all other punctures on $S$ which exit both ends of the cylinder. We move the distinguished puncture $p$ so that it lies to the right of $p_{-1}$ and to the left of $p_0$.
We also choose one non-isolated end of $S$ to correspond to the left direction (the accumulation point of $p_i$ for $i < 0$) and one to correspond to the right direction (the accumulation point of $p_i$ for $i > 0$), which gives a well-defined notion of a front and back of the cylinder for $S$. 

Just as in \cite{AMP}, we fix a complete hyperbolic metric on $S$ and let $B_0$ be a horocycle at a height sufficiently far out the cusp corresponding to $p_0$. Fix a shift map $h$ on $S$ whose domain contains exactly the collection $\{p_i\}$ for $i \in \mathbb{Z}$ and which shifts $p_i$ to $p_{i+1}$ for all $i \in \Z$.

\begin{defn}\label{def:B_i}
Define the simple closed curves $B_i:=h^i(B_0)$ for $i\in \mathbb Z$.  Then $B_i$ is a simple closed curve bounding the puncture $p_i$.  
We identify each $B_i$ with $\mathbb S^1$ and fix the north pole of each $B_i$.
\end{defn}

\subsection{Coding arcs}\label{sec:thecode}

Suppose $\gamma$ is an oriented arc on $S$ starting and ending at $p$.   We code $\gamma$ exactly as in \cite{AMP}. For the sake of brevity, we give the following examples of arcs and their codes instead of discussing the code in detail.

\begin{ex}  \label{ex:codes}
Consider the arcs shown in Figure \ref{fig:examplecodes}.  The elements $k\in\Z$ shown under $S$ denote the subscript on the simple closed curves $B_k$.  The code for $\alpha$ is $P_s0_o1_u2_o2_u1_u0_uP_s$, the code for $\beta$ is  $P_sP_uP_o0_o1_o2_o2_u1_o0_oP_s$, the code for $\gamma$ is given by $P_s(-1)_o(-2)_o(-2)_u(-1)_uP_u0_u1_u1_o0_oP_s$, and the code for $\delta$ is given by $P_s(-1)_oC(-2)_o(-2)_uC(-1)_oP_s$. Note that $P_s$ indicates that the arc starts or ends at the puncture $p$, the subscript $o/u$ corresponds to whether the arc passes over of under that puncture, and the $C$ in the code for $\delta$ denotes the fact that $\delta$ goes to the back of the surface $S$. 
\end{ex}

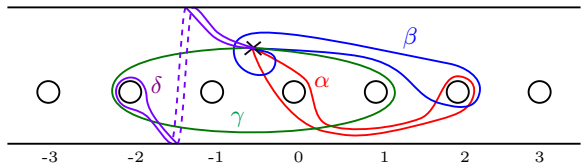
\begin{figure}[h]
\begin{center}
\begin{overpic}[width=3in]{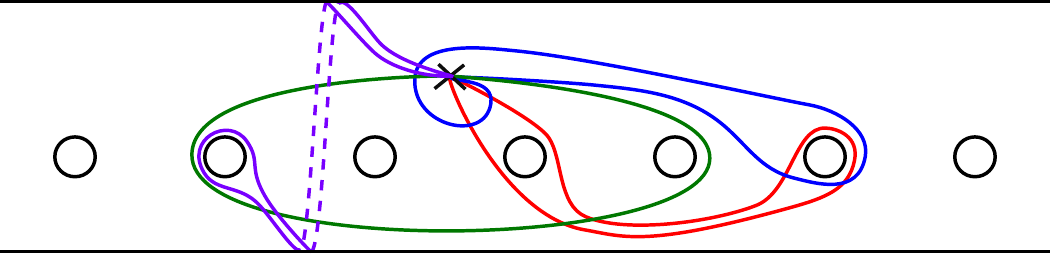}
\put(53.5,10){\textcolor{red}{\small$\alpha$}}
\put(69,17.5){\textcolor{blue}{\small$\beta$}}
\put(39,4){\textcolor{ForestGreen}{\small$\gamma$}}
\put(25,9){\textcolor{violet}{\small$\delta$}}
\put(6, -3){\tiny{-3}}
\put(21, -3){\tiny{-2}}
\put(35, -3){\tiny{-1}}
\put(50, -3){\tiny{0}}
\put(65, -3){\tiny{1}}
\put(79, -3){\tiny{2}}
\put(92, -3){\tiny{3}}
\end{overpic}
\caption{Arcs on the front of the surface $S$, whose codes are given in Example \ref{ex:codes}.  The $\times$ denotes the puncture $p$, and the elements $k\in\Z$ shown under $S$ denote the subscript on the simple closed curves $B_k$. }
\label{fig:examplecodes}
\end{center}
\end{figure}

The appearance of repeated characters in the code of an arc indicates backtracking in the arc so that we have the following. 

\begin{defn}\label{def:reducedcode}
Let $\gamma$ be an oriented arc on $S$ starting and ending at $p$.  A code for $\gamma$ is \textit{reduced} if no two adjacent characters are the same and if the character immediately following the initial $P_s$ or preceding the terminal $P_s$ is not $P_{o/u}$.
\end{defn}

\noindent Note that if a triple appears in the code for an arc, it is reduced to a single character according to our convention, as only \emph{pairs} of repeated characters are removed.

\begin{defn}\label{def:codelength}
The \textit{code length} of an arc $\gamma$, denoted $\ell_c(\gamma)$, is the number of characters in a reduced code for $\gamma$. 
\end{defn}

Given a string of characters $\alpha=a_1a_2\ldots a_n$, we denote by $\overline{\alpha}$ the reverse of $\alpha$, so that $\overline{\alpha}=a_na_{n-1}\ldots a_2a_1$.  If $\alpha$ is an arc, then  $\overline{\alpha}$ is the same arc with the opposite orientation.

\subsection{Spokes}
We now introduce spokes, which are special segments on $S$. 
\begin{defn} \label{def:backloop}

A \textit{segment} is a simple path with at least one endpoint which is not a puncture, and no endpoints on a puncture other than $p$.  We code a segment in an analogous way as we do arcs. 
\end{defn}

Given any essential, separating simple closed curve $\chi$ on the front of $S$ such that one connected component of $S \setminus \chi$ is a finite-type surface containing the puncture $p$, we let $g_{\chi,j}=hT_\chi^j$. We call the connected component of $S \setminus \chi$ containing $p$ the \emph{interior} or \emph{inside of $\chi$} and the other connected component is the \emph{outside of $\chi$}. Below, we prove some of the technical results in the case where $j = 1$ for brevity, which is actually the most difficult case, and refer to $g_{\chi,j}$ as $g_\chi$ or $g$ for notational simplicity. All proofs generalize easily to $j > 1$, and in fact simplify a bit. 

\begin{defn}\label{def:standard}
A simple closed curve $\chi$ on the front of $S$ is \emph{standard} if it has the form in Figure \ref{fig:sectors}; we assume that $\chi$ contains the puncture $p$ in its interior and that  the right-most puncture contained in $\chi$ is $p_0$. In addition, there are no punctures ``above" $\chi$.  The punctures on the interior of $\chi$ are called \textit{interior punctures}.
\end{defn}

\noindent We first consider curves that can be translated to standard curves by powers of $h$. 

\begin{lem}\label{lem:standardscc}
Let $\chi$ be a simple closed curve on $S$ such that $\chi'$ is homotopic to $ h^i(\chi)$ which is standard for some $i$.  Then $g_\chi$ and $g_{\chi'}$ are conjugate by a power of $h$, and thus, $g_\chi$ is loxodromic with respect to the action of $\Map(S,p)$ on $\mc A(S,p)$ if and only if $g_{\chi'}$ is. 
\end{lem}

\begin{proof}
Fix any simple closed curve $\chi$ on the front of $S$ containing the puncture $p$. 
If $h^i(\chi) = \chi'$ is standard, we see that \[h^{i}(g_\chi)h^{-i}=h^{i}(hT_{\chi})h^{-i}=h(h^{i}T_{\chi}h^{-i})=hT_{h^{i}(\chi)}=hT_{\chi'}=g_{\chi'},\] which concludes the proof.  
\end{proof}

We therefore assume in the remainder of the paper that $\chi$ is standard. Moreover, we choose the homotopy representative of $\chi$ to contain no backtracking. It will be useful to have a code for $\chi$, even though it is not a arc.  We define this code in the usual way, by tracking whether $\chi$ passes over or under each puncture, always assuming that $\chi$ is oriented clockwise.  However, since $\chi$ does not have a well-defined starting point, such a code is only well-defined up to cyclic permutations.  This will cause no problems in this paper.  For example, a code for the curve $\chi$ shown in Figure~\ref{fig:EasyCase} is $P_o0_o0_uP_u{-1}_u{-2}_u{-2}_o{-1}_oP_o$.

\begin{defn}
For each $k$ such that $p_k$ is contained inside $\chi$, and for $p_1$, a \emph{spoke} (to $p_k$) is a segment whose initial point is $p$ and whose terminal point is the north pole of $B_k$, as in Figure \ref{fig:sectors}.  In particular, such a segment passes over all punctures contained in $\chi$ between $p$ and $p_k$.  We label the spokes $\sigma_i$ consecutively, starting from the right; if there are $\ell$ punctures on the interior of $\chi$, then the spokes are $\sigma_0,\sigma_1,\dots, \sigma_\ell$.  Note that since $\chi$ is standard, $\sigma_{0}$ is the only spoke whose terminal point is outside of $\chi$.

Given a spoke $\sigma_i$, let $P(i) \in \mathbb{Z}$ be the index so that $\sigma_i$ is a segment from $p$ to $B_{P(i)}$, i.e., $\sigma_i$ ends at the simple closed curve corresponding to $P(i)$. If $P(i)\leq -1$, define a code for $\sigma_i$ to be $P_s(-1)_o(-2)_o\dots P(i)_s$.  Here, $P(i)_s$ indicates that the segment stops at a point on $B_{P(i)}$.  If $P(i)=0$, the code for $\sigma_i$ is $P_s0_s$, while if $P(i)=1$, the code for $\sigma_i$ is $P_s0_o1_s$.
\end{defn}

\begin{defn}The spokes divide the interior of $\chi$ into regions which we call \textit{sectors}.  We denote the sector bounded by $\sigma_i$ and $\sigma_{i+1}$ by $S_{i}$ for $i = 0, \ldots, \ell -1$ and $S_\ell$ is the sector bounded by $\sigma_\ell$ and $\sigma_0$.  See Figure \ref{fig:sectors}.   
\end{defn}

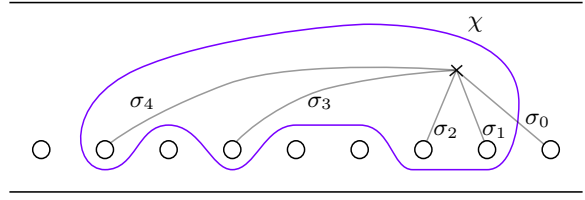
\begin{figure}
    \centering
    \begin{overpic}[width=3in,height=1in]{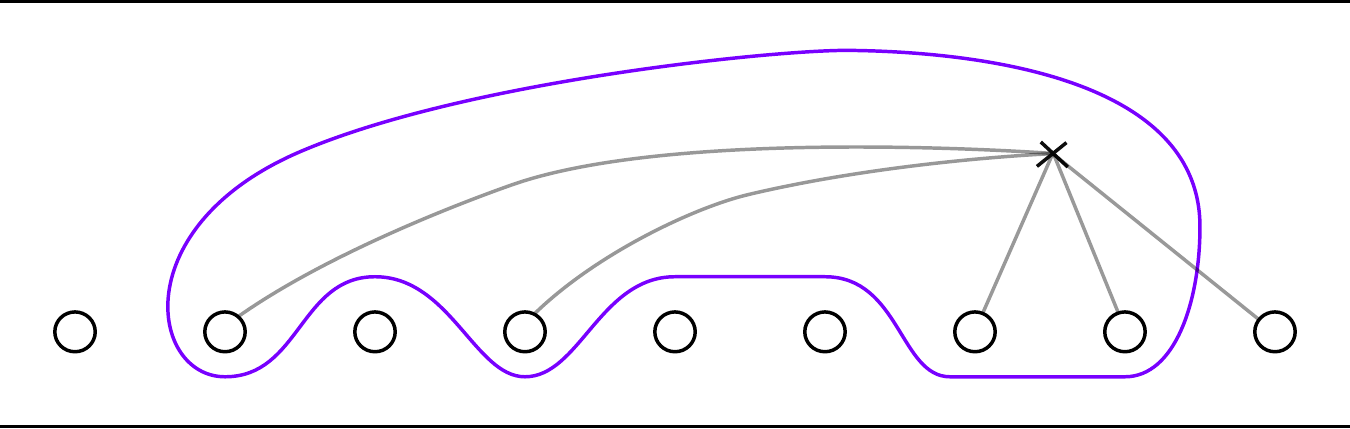}
\put(80,29){\small$\chi$}    
\put(90,12){\small$\sigma_0$}
\put(82.5,10){\small$\sigma_1$}
\put(74,10){\small$\sigma_2$}
\put(52,15){\small$\sigma_3$}
\put(21,15){\small$\sigma_4$}
\end{overpic}
    \caption{The curve $\chi$ is a standard simple closed curve  on the front of $S$ containing the puncture $p$.  The purple segments are spokes, and the sectors are labeled in green.}
    \label{fig:sectors}
\end{figure}

\begin{defn}
If $i\ge 1$,  an arc $\delta$ starting at $p$ \textit{initially follows a spoke $\sigma_i$} if an initial portion of the reduced code for $\delta$ agrees with the code for $\sigma_i$ with the last character replaced with either $P(i)_{u}$ or $P(i)_{o}P(i)_{u}$.
Similarly, we say that an arc $\delta$ \textit{initially follows the spoke $\sigma_0$} if the initial portion of its reduced code begins with $P_s0_o1_u$.
\end{defn}

\begin{ex}\label{ex:FollowingSpokes}
Consider the arcs $\beta_1, \beta_2,\beta_3,\beta_4$ in Figure \ref{fig:FollowingSpokes}.  The arc $\beta_1$ starts in sector $S_1$ and initially follows $\sigma_1$,  the arc $\beta_2$ starts in sector $S_2$ and initially follows $\sigma_{3}$,  the arc $\beta_3$ starts in sector $S_4$ and  initially follows  no spoke,  while $\beta_4$ starts in $S_1$ and initially follows no spoke.
\end{ex}

\begin{figure}
\centering
\begin{overpic}[width=4in,height=1in]{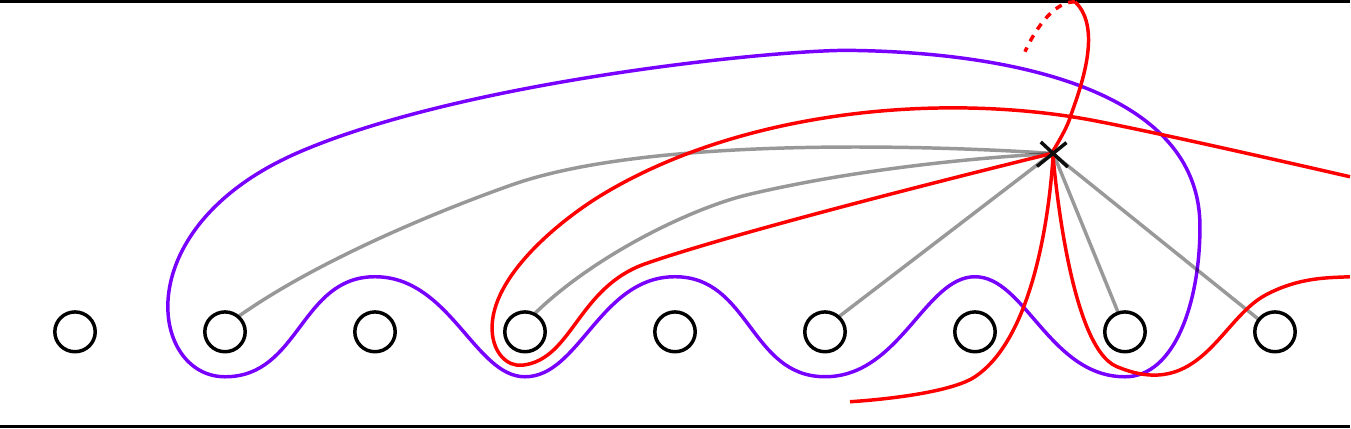}
\put(85,12){\tiny$\sigma_0$}
\put(81.5,10){\tiny$\sigma_1$}
\put(70,10){\tiny$\sigma_2$}
\put(41,10){\tiny$\sigma_3$}
\put(21,10){\tiny$\sigma_4$}
\put(95,16.5){\textcolor{red}{\small$\beta_2$}}
\put(95,9.5){\textcolor{red}{\small$\beta_1$}}
\put(95,16.5){\textcolor{red}{\small$\beta_2$}}
\put(81,21.5){\textcolor{red}{\small$\beta_3$}}
\put(74,2){\textcolor{red}{\small$\beta_4$}}
\end{overpic}
\caption{The arcs in Example~\ref{ex:FollowingSpokes}.}
\label{fig:FollowingSpokes}
\end{figure}
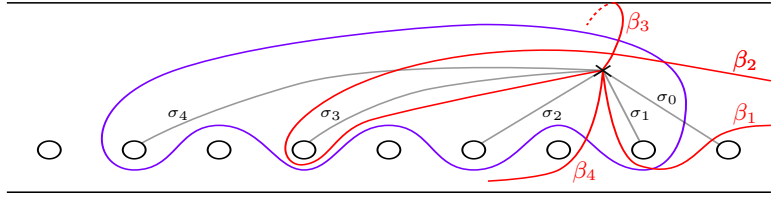

\subsection{Standard position}
Every arc $\gamma$ with reduced code is homotopy equivalent to an arc $\gamma'$ with the same code that satisfies the following properties.
\begin{itemize}
    \item There is an initial segment $\gamma_1'$ of $\gamma'$ which is contained in a unique sector.
    \item If $\gamma_1'$ is contained in $S_i$ with $i\geq 1$, then $\gamma_1'$ begins with $P_s(-1)_o\dots P(i)_o$.  If $i=0$, then the arc begins with $P_s0_o$.
    \item If $\gamma_1'$ is followed by a segment that agrees with the code for $\chi$, then  $\gamma'$ does not intersect $\chi$ until after this segment.
    \item If the character after $\gamma_1'$ does not agree with the code for $\chi$,  then the arc crosses $\chi$ immediately after traversing $\gamma'_1$.
\end{itemize}

If $\delta$ is an arc that begins in sector $S_i$ and does not follow $\sigma_{i+1}$ or $\sigma_i$, then we say $\delta$ \textit{exits immediately}.  Such a $\delta$ can exit immediately in two ways.  First, it could go to the back directly after crossing $\chi$. Otherwise, there must be an exterior puncture $q$ between $P(i+1)$ and $P(i)$, and $\delta$ must contain $q_u$.
For instance, in Figure \ref{fig:FollowingSpokes} the former way occurs for $\beta_3$ and the latter way occurs for $\beta_4$.


\section{Blocking cancellation}

 Suppose $\delta$ begins in $S_i$ and initially follows $\sigma_{i+1}$.  If $i\ge 1$, then by definition a code for $\delta$ begins with
$P_s(-1)_o\dots (P(i+1)+1)_oP(i+1)_u.$ 
If $i=0$, then  this code begins with $P_s0_o0_u$.
In particular, because the final character of this code in either case agrees with that of $\chi$, the arc $\delta$ must follow $\chi$ clockwise for at least one character.
We say $\delta$ \textit{hooks a puncture} if 
either: (a) $\delta$ follows $\chi$ clockwise for less than one full turn and then has a character $k_{o/u}$ that disagrees with the code for $\chi$ such that $p_k$ is an interior puncture; or (b) if $\delta$ follows $\sigma_{i+1}$, loops around $P(i+1)$, and then follows $\sigma_{i+1}$ backwards to $p$. In (a), we say $\delta$ hooks the puncture $p_k$, while in (b) we say $\delta$ hooks the puncture $P(i+1)$. 
See Figure \ref{fig:HookAPuncture}.  Notice that in (a), $k_{o/u}$ is the first character after $\delta$ leaves $S_i$ that disagrees with the code for $\chi$, so we can think of this as the point where $\delta$ stops fellow traveling $\chi$.

\begin{figure}
    \centering
    \begin{overpic}[width=3in]{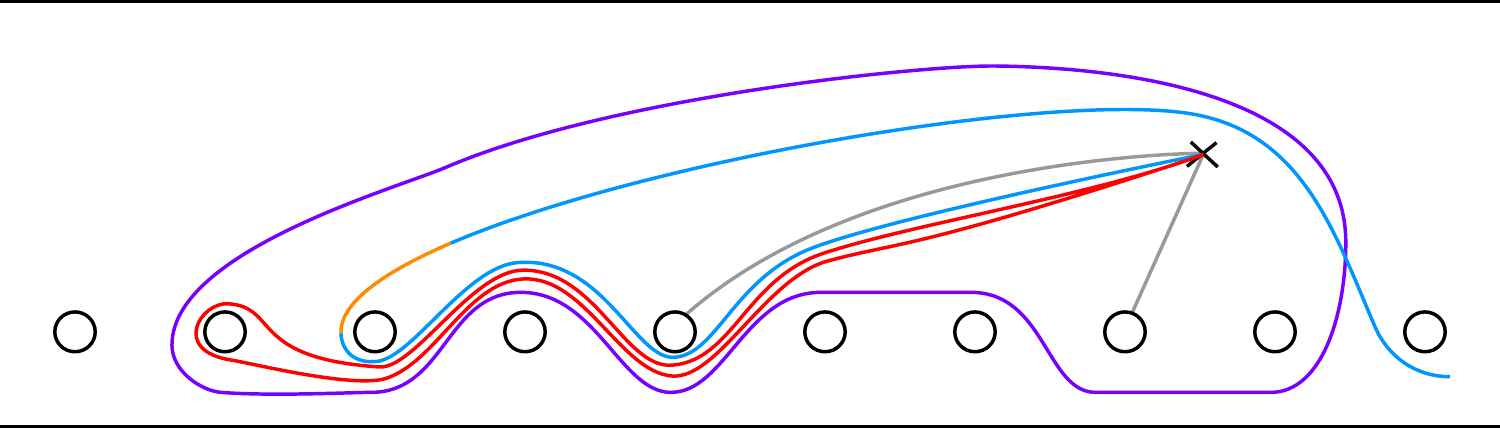}
    \put(35,15){\textcolor{cyan}{\small$\delta$}}
    \put(15,9){\textcolor{ForestGreen}{\tiny$T_\chi(\delta)$}}
    \end{overpic}
    \vspace{.15in}
    
    \begin{overpic}[width=3in]{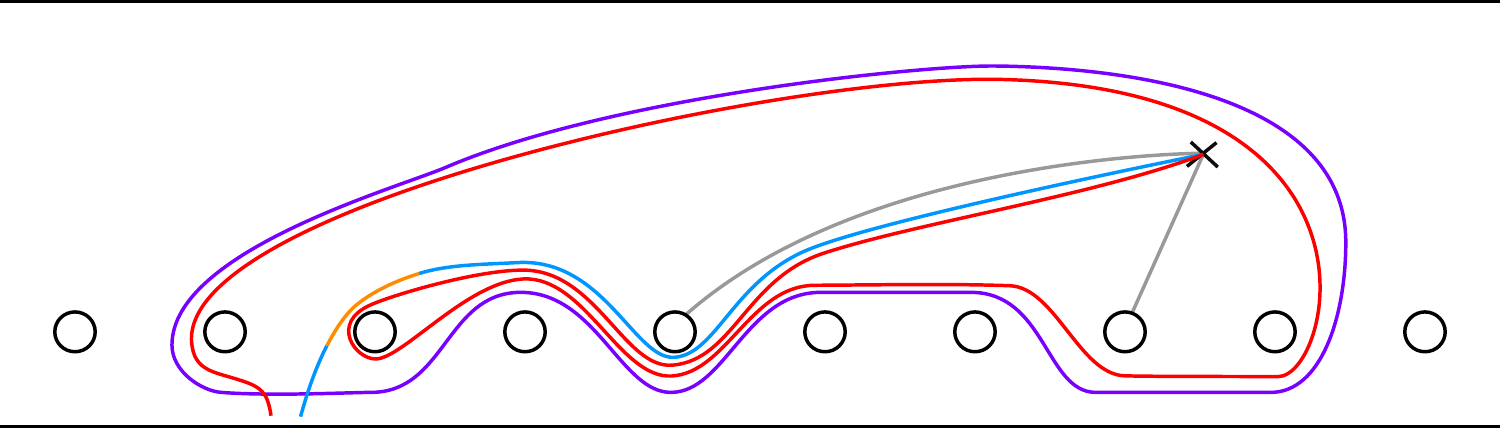}
    \put(40,9.5){\textcolor{cyan}{\small$\delta$}}
    \put(29,14.5){\textcolor{red}{\tiny$T_\chi(\delta)$}   }
    \end{overpic}
    \vspace{.15in}    

    \begin{overpic}[width=3in]{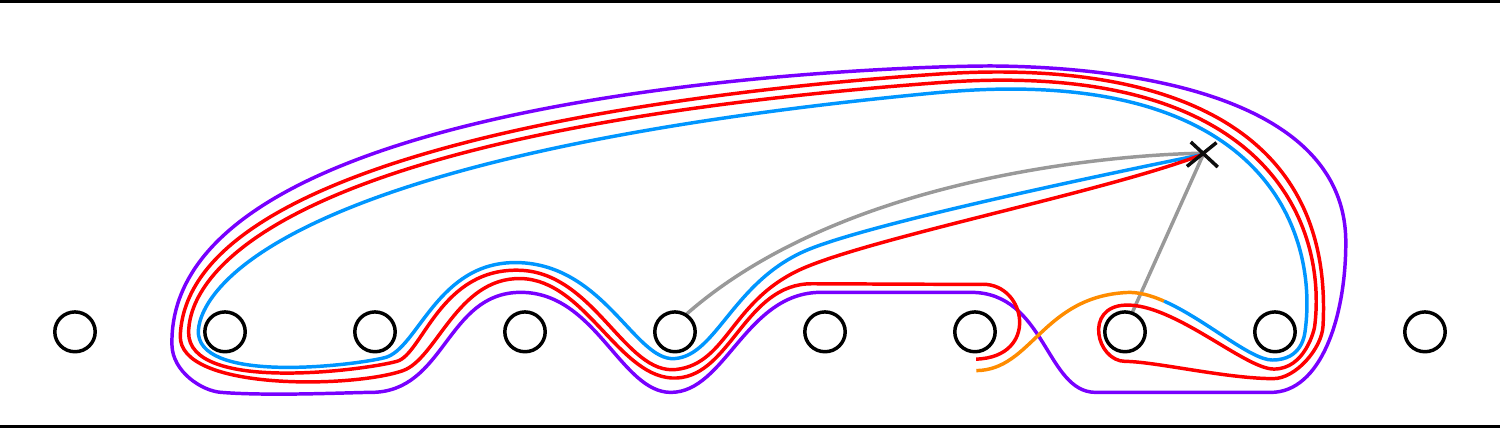}
    \put(40,9.5){\textcolor{cyan}{\small$\delta$}}
    \put(69,13){\textcolor{red}{\tiny$T_\chi(\delta)$} }
    \end{overpic}
    \vspace{.15in}
    
    \begin{overpic}[width=3in]{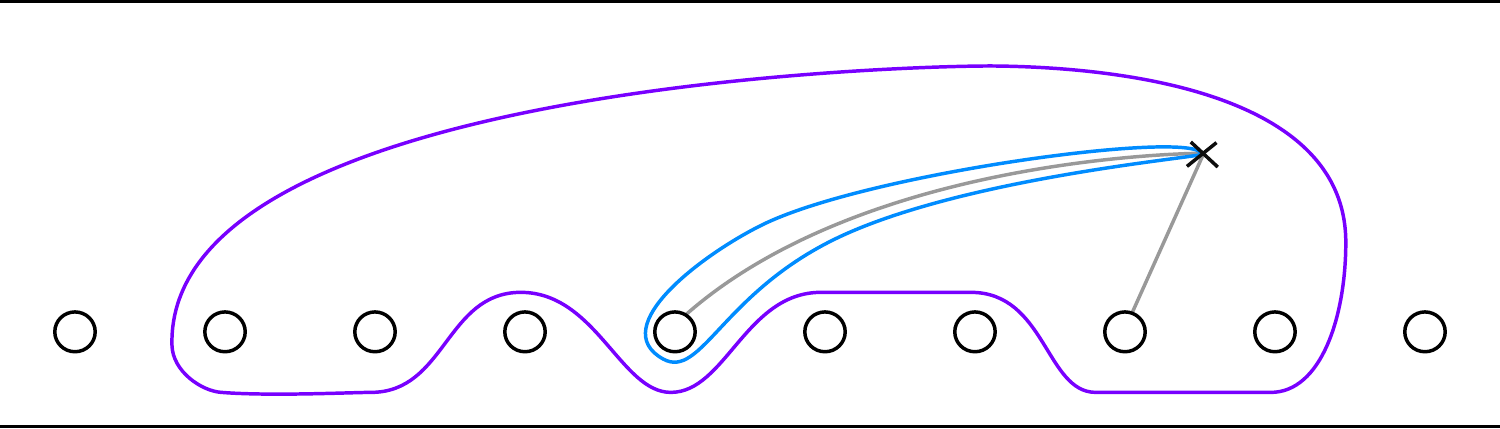}
    \put(55,20){\small$\textcolor{cyan}{\delta=T_\chi(\delta)}$}
    \end{overpic}
     \caption{Four ways an arc $\delta$ that initially follows $\sigma_{i+1}$ can hook a puncture.  The first three pictures fit (a), while the last fits (b).  The orange highlighted segments in each of the first three correspond to the character $k_{o/u}$ in the definition.  In the last panel, $\delta$ is fixed by $T_\chi$.  This illustrates part (1) of Lemma~\ref{lem:blocker}.}
    \label{fig:HookAPuncture}
\end{figure}
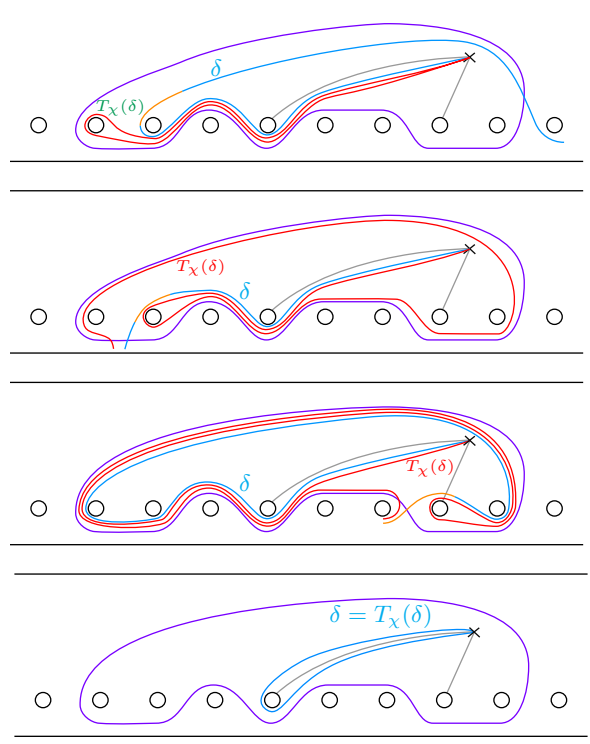

Our strategy will be to understand how the images of arcs begin after iteratively applying $T_\chi$ and $h$.  In order to do this, we need to understand what kind of cancellation can occur when we apply $T_\chi$ or $h$ to an arc. 
The following lemma gives specific conditions under which there is no cancellation between a particular initial segment of $\delta$ and the remainder of $\delta$.  In the following sections, we will focus only on initial segments of arcs that fit into one of these categories.
\begin{lem}\label{lem:blocker}
If $\delta$ is an arc, then the following hold:
	\begin{enumerate}
	\item If $\delta$ hooks a puncture $p_k$, then there is no cancellation between the images under $T_\chi$ of the portion of $\delta$ before $p_k$ and the portion after $p_k$.
	\item Suppose $\delta$ begins in a sector $S_i$ and then $\delta$ either 
	    \begin{itemize}
	        \item  exits immediately, so that the character in a reduced code for $\delta$ immediately after crossing $\chi$ is either $C$ or $q_u$ for some exterior puncture; or 
	        \item follows $\chi$ counterclockwise for at least one character $k_{o/u}$.
	    \end{itemize} 
	    Then there is no cancellation between the images under $T_\chi$ of the portion of $\delta$ before this character and the portion after this character. 
	\item If $\delta$ contains the character $k_{o/u}$ for any $k\in [1,\infty)$, then there is no cancellation between the images under $g$ of the portion of $\delta$ before $k_{o/u}$ and the portion after $k_{o/u}$.
	\item If $\delta$ contains the character $k_{o/u}$ for any $k\in [2,\infty)$, then there is no cancellation between the images under $g^{-1}$ of the portion of $\delta$ before $k_{o/u}$ and the portion after $k_{o/u}$.
	\end{enumerate}
\end{lem}

The proof of this lemma is intuitively straightforward: in each case, the character that blocks cancellation disagrees with the code for $\chi$.   
When applying $T_\chi$ to an arc or segment, a copy of the code for $\chi$ is inserted into the code for $\delta$, representing the twisting around $\chi$. Thus, the only cancellation that can occur in $T_\chi(\delta)$ is when the copy of $\chi$ is inserted into the code for $\delta$ next to a subsegment of $\overline{\chi}$. In (1) and (2), it is clear that the initial segment does not end with a subsegment of $\overline{\chi}$. 
For (3), the additional intuition is that $T_\chi$ fixes the character $k_{o/u}$ because it is not contained in a code for $\chi$, and then $h$ is simply a shift, which does not cause any additional cancellation. For (4), the reasoning is reversed: $h^{-1}$ does not cause any cancellation, and then $T_\chi^{-1}$ fixes the character $h^{-1}(k_{o/u})=(k-1)_{o/u}$, since $k\geq 2$.

Rather than giving a lengthy and technical proof of these straightforward facts, we illustrate the proof of each case; see Figure~\ref{fig:HookAPuncture} for case (1) and Figure~\ref{fig:BlockingCancellation} for cases (2) and (3).

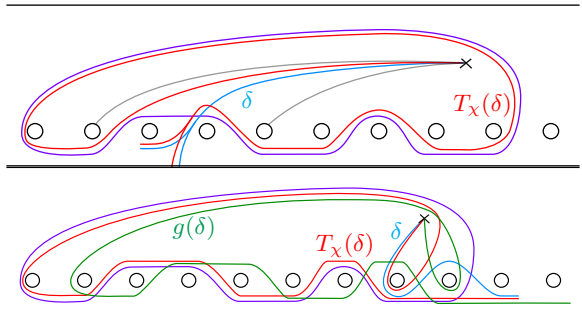
\begin{figure}
    \centering
    \begin{overpic}[width=3in]{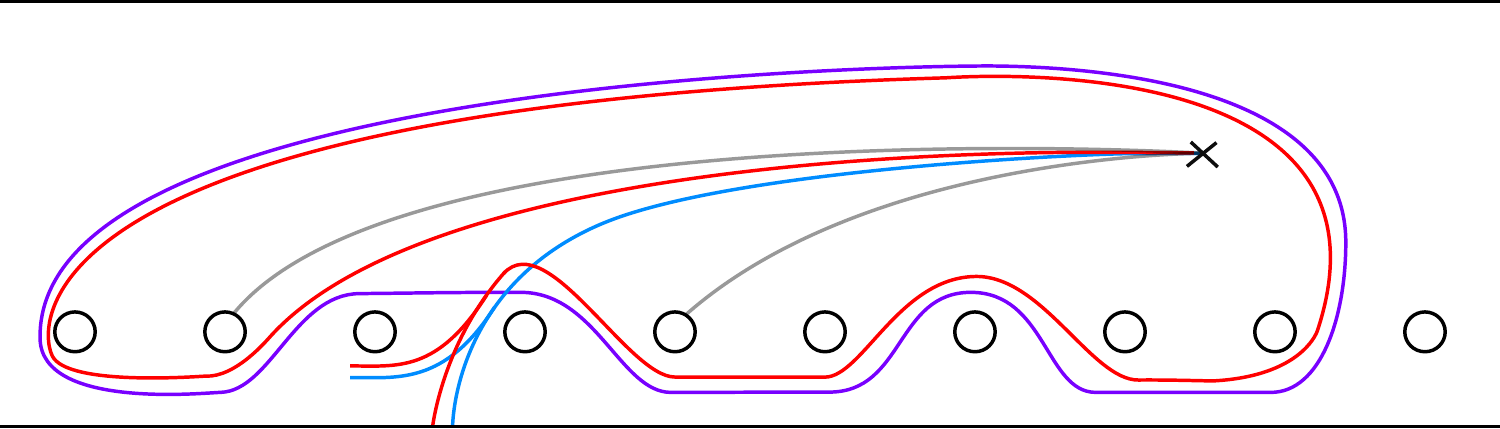}
    \put(41,10.5){\textcolor{cyan}{\small$\delta$}}
    \put(78,10){\textcolor{red}{\small$T_\chi(\delta)$}}
    \end{overpic}
\vspace{.15in}
   \begin{overpic}[width=3in]{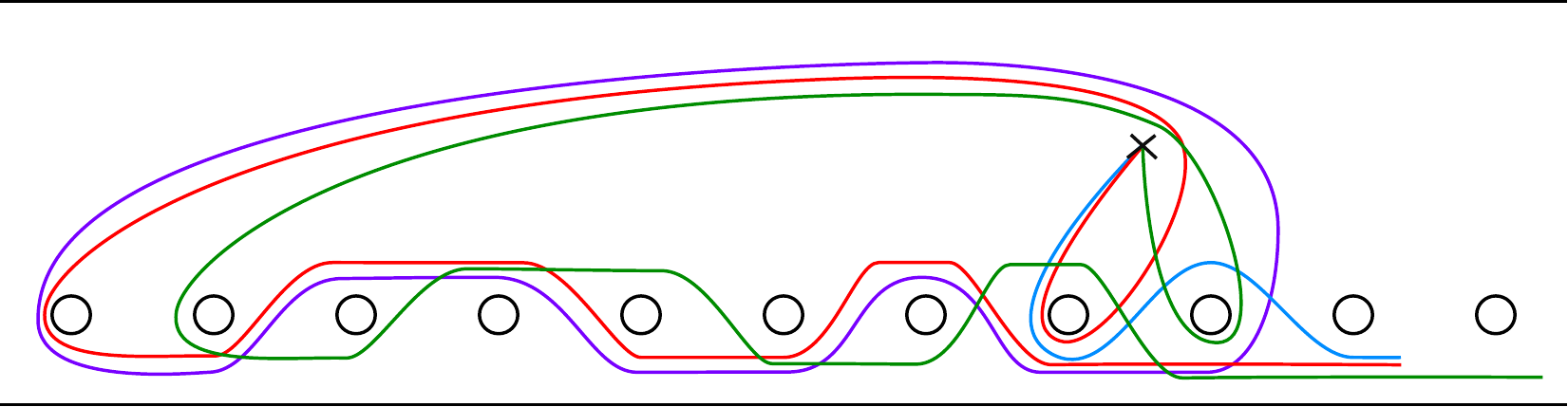}
   \put(67,13){\textcolor{cyan}{\small$\delta$}  }
   \put(29,14){\textcolor{ForestGreen}{\small$g(\delta)$}}
   \put(54,11){\textcolor{red}{\small$T_\chi(\delta)$}   }
    \end{overpic}
    \caption{Illustration of parts (2) and (3) of Lemma~\ref{lem:blocker}.  Note that in Case (3), the only cancellation occurs in the green highlighted portion.}
    \label{fig:BlockingCancellation}
\end{figure}

\section{A ``starts like" function}

In the remainder of the paper, we will construct an explicit quasi-geodesic axis in $\mathcal{A}(S,p)$ on which $g$ acts by translations and therefore show that $g$ acts loxodromically on $\mathcal{A}(S,p)$.
To accomplish this, we will follow a strategy similar to that in \cite{Bavard, AMP}, in which  a ``starts like" function is constructed and used  to bound distance in $\mathcal{A}(S,p)$ from below.

We begin by defining the vertices in the quasi-axis for $g$ via the following formula
\begin{equation}\label{eqn:defofalpha}
\alpha_k=\begin{cases}
P_s0_o1_u2_o3_o3_u2_o1_u0_oP_s,&k=0\\
g^k(\alpha_0),&k\neq 0
\end{cases}
\end{equation}
where $k\in\Z$.

\subsection{Defining a ``starts like" function}
The following two definitions were first introduced in \cite{AMP} and inspired by \cite{Bavard}; here they are modified to fit our current situation.

\begin{defn}\label{defn:beginning}
Given any arc $\gamma$ on $S$, we define the \emph{beginning of $\gamma$}, denoted by $\mathring\gamma$, to be the first $\lfloor\frac{1}{2} \ell_c(\gamma)\rfloor -2$ characters in a reduced code for $\gamma$. Recall that code length was defined in Definition \ref{def:codelength}.
\end{defn} 
\begin{defn}\label{def:startslike}
Fix any $k\in\Z_{\ge 0}$. An arc $\delta$ \textit{starts like} $\alpha_k$ if the maximal initial or the maximal terminal segment of $\delta$ which agrees with an initial or terminal segment of $\alpha_k$ (not necessarily respectively) has code length at least $\lfloor\frac12\ell_c(\alpha_k)\rfloor-2$. 
\end{defn}

The following lemma is a straightforward application of Definition \ref{defn:beginning} and the definition of $\alpha_0$ in Equation \eqref{eqn:defofalpha}.

\begin{lem}
If $\delta$ is disjoint from $\mathring\alpha_0$, then $\delta$ does not contain an instance of $\chi$ or $\overline{\chi}$ in its reduced code.
\end{lem}

Consider the map \begin{equation}\label{eqn:startslike}
    \phi\colon \mc A(S,p) \to \Z_{\geq 0},
\end{equation} defined by setting $\phi(\delta)$ equal to the largest $i\geq 0$ such that $\delta$ starts like $\alpha_i$.  If no such $i$ exists, then set $\phi(\delta)=0$.
Using this function, we have the following simple lemma, where we do not require that $j$ is non-negative.

\begin{lem}\label{lem:shiftstartslike}
Suppose an arc $\delta$ starts like $\alpha_i$ for some $i\geq 1$.  Then for any $j$ such that $i+j\geq 0$, the arc $g^j(\delta)$ starts like $\alpha_{i+j}$.
\end{lem}

\begin{proof}
 First, notice that the arc $\alpha_0$ is symmetric in the sense that \[\alpha_0=P_s0_o1_u2_o3_o3_u2_o1_u0_oP_s=\tau3_o3_u\overline{\tau}.\]  By Lemma~\ref{lem:blocker}(3), the characters $3_o$ and $3_u$  block cancellation.  In particular, there is no cancellation between the first half of $\alpha_0$ and the second half. Moreover,  any cancellation that occurs in $\tau3_o$, the first half of $\alpha_0$, must also occur in $3_u\overline{\tau}$, the second half of $\alpha_0$. Thus the image of the first half of $\alpha_0$ is the first half of $\alpha_1$.  Since the characters $1_u$ and $2_o$ also block cancellation by Lemma~\ref{lem:blocker}(3) and the last character of $\mathring{\alpha}_0$ is $1_u$, it follows that $g(\mathring{\alpha}_0)=\mathring{\alpha}_1$ and that the last character of $\mathring{\alpha}_1$ is $2_u$.  Inductively, we conclude that $g(\mathring{\alpha}_i)=\mathring\alpha_{i+1}$ and the last character of $\mathring\alpha_{i+1}$ is $(i+1)_u$.

To show the lemma, it suffices to prove the result in the special case that $j=\pm 1$ and $i+j\ge 0$. An inductive argument then exhibits the general case.
To this end, if $j=1$ and $\delta$ starts like $\alpha_i$, then the initial subsegment of $\delta$ is given by $\mathring\alpha_i$.  Moreover, the terminal character $(i+1)_u$ of $\mathring{\alpha}_{i}$, which must appear in $\delta$,  blocks cancellation by Lemma \ref{lem:blocker}(3).
In particular 
  $g(\mathring\alpha_i)=\mathring\alpha_{i+1}$ is the initial subsegment of $g(\delta)$, as required. 
Similarly, if $j=-1$, then since $i+j\ge 0$, we have $i+1\ge 2$, and therefore $(i+1)_u$ blocks cancellation by Lemma \ref{lem:blocker}(4). 
In particular, $g^{-1}(\delta)$ will contain $\mathring\alpha_{i-1}$, as required.
\end{proof}

\subsection{A lower bound on the ``starts like" function}

The goal of this section is to prove the existence of a uniform $M$ so that given any arc $\delta$ disjoint from $\mathring{\alpha}_0$, there exists $0 \leq k \leq M$  such that $g^k(\delta)$ starts like $\alpha_1$.  We begin with an example that motivates the method of proof; see Figure~\ref{fig:EasyCase}.  The arc $\delta$ shown in the figure begins in sector $S_3$ and follows $\sigma_3$.  An initial subpath is invariant under $T_\chi$, so  $g(\delta)$ begins in $S_2$ and follows $\sigma_2$. Applying $g$ two more times yields an arc $g^3(\delta)$ that begins in $S_0$ and follows $\sigma_0$.  The following lemma shows that  applying $g$ one final time ensures that $g^4(\delta)$ begins like $\alpha_1$. 

\begin{figure}
\centering
\begin{overpic}[width=3in,height=1in]{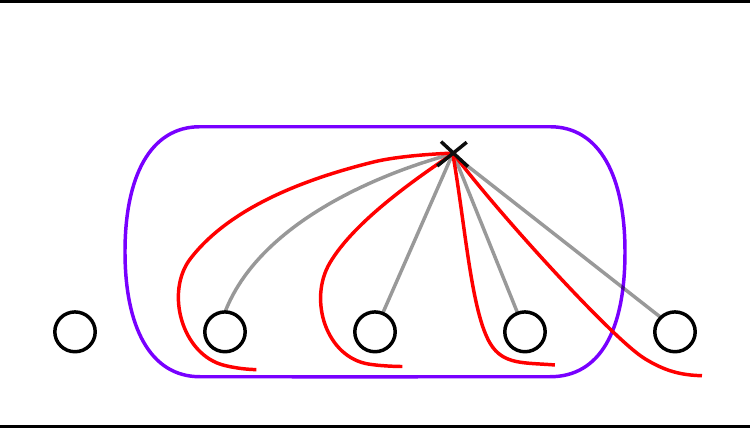}
\put(21,10){\tiny$\delta$}
\put(36,10){\tiny$g(\delta)$}
\put(54,10){\tiny$g^2(\delta)$}
\put(80,2){\tiny$g^3(\delta)$}
\end{overpic}
\caption{The most straightforward situation: after applying $g$ 3 times, we have an arc that begins in $S_0$ and follows $\sigma_0$.}
\label{fig:EasyCase}
\end{figure}
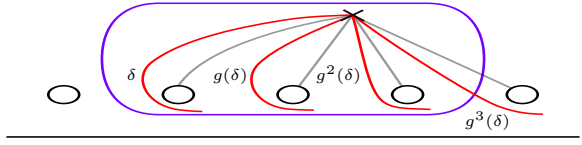

\begin{lem}\label{lem:sigma0toalpha1}
Suppose $\delta$ is an arc which begins in sector $S_0$ and initially follows $\sigma_0$.  Then $g(\delta)$ starts like $\alpha_1$.
\end{lem}

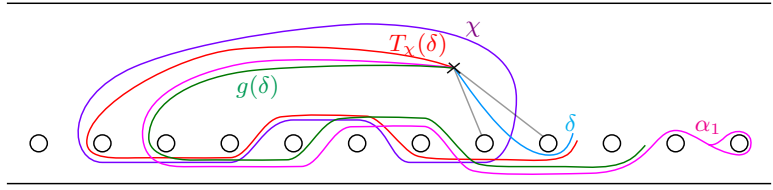
\begin{figure}
    \centering
    \begin{overpic}[width=4in]{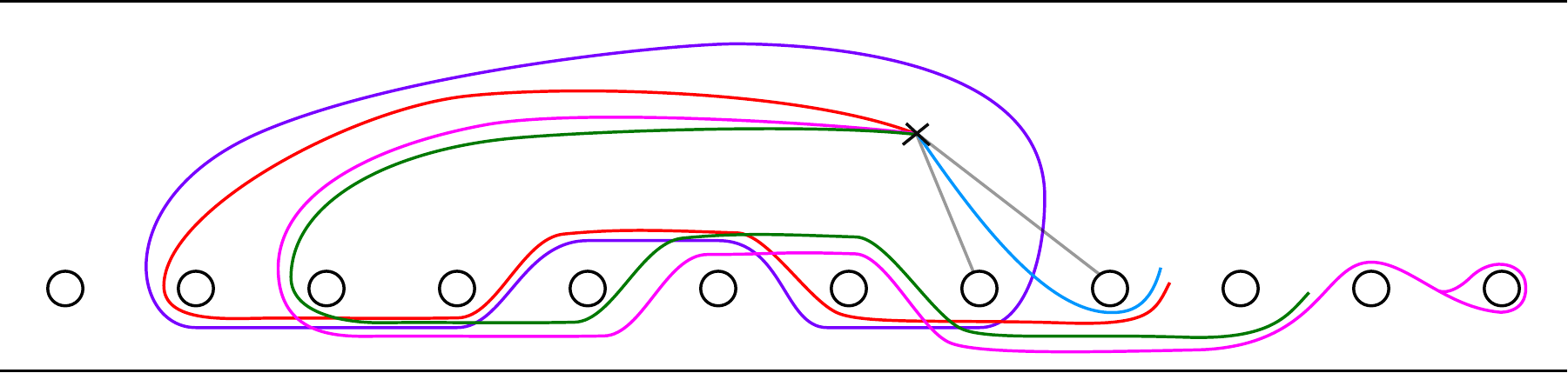}
    \put(60,20){\textcolor{violet}{\small$\chi$}}
    \put(50,17.5){\textcolor{red}{\small$T_\chi(\delta)$}}
    \put(30,12){\textcolor{ForestGreen}{\small$g(\delta)$}}
    \put(90,7){\textcolor{Magenta}{\small$\alpha_1$}}
    \put(73,7){\textcolor{cyan}{\small$\delta$}}
    \end{overpic}
    \caption{An example of an arc $\delta$ as in Lemma~\ref{lem:sigma0toalpha1}.  The arc $g(\delta)$ begins by following $\mathring\alpha_1$, and so $\phi(g(\delta))\geq 1$.}
    \label{fig:sigma0toalpha1}
\end{figure}

\begin{proof}
Since $\delta$ begins in sector $S_0$ and initially follows $\sigma_0$, a reduced code for $\delta$ is $P_s0_o1_u\delta'$ for some $\delta'$.  We may compute the image of $P_s0_o1_u$ directly and see that it agrees with the code for $\mathring{\alpha}_1$; see Figure~\ref{fig:sigma0toalpha1}.  Thus it remains to show that there is no cancellation between $g(P_s0_01_u)$ and $g(\delta')$ in $g(\delta)$.  The initial character of $\delta'$ must be either $1_o$, $2_{o/u}$, or $C$. All of these block cancellation by Lemma~\ref{lem:blocker}. 
\end{proof}

The remainder of this section shows that
if $\delta$ is \textit{any} arc that is disjoint from $\mathring\alpha_0$, applying $g$ sufficiently many times results in an arc which begins in $S_0$ and follows $\sigma_0$ so that applying $g$ one more time results in an arc that starts like $\alpha_1$.  Moreover, we will show that we only need to apply $g$ at most $3\ell+2$ times for this behavior to occur, where $\ell$ is the number of punctures in the interior of $\chi$ that are not $p$.

From our standard position, we will assume for the time being that every arc begins in a sector $S_i$ for some $i\in\{1,\dots, \ell\}$; the case $i=0$ will be handled later. Given 
such an arc $\delta$, there are three possibilities: $\delta$ either exits immediately, follows $\sigma_i$, or follows $\sigma_{i+1}$.  We first show that either $g(\delta)$ or $g^2(\delta)$ must begin in sector $S_{i-1}$, that is, after applying $g$ at most 2 times, the image of $\delta$ has moved into the sector to the right.  The first two cases are dealt with in the following lemma, while the third case is dealt with in Lemma~\ref{lemma:casec}.
\begin{lem}\label{lemma:caseab}
Let $\delta$ be an arc which begins in sector $S_i$ for some $i\in \{1,\dots,\ell\}$.
Suppose that either:
\begin{enumerate}[(a)]
\item $\delta$ exits immediately, or
\item $\delta$ follows $\sigma_i$.
\end{enumerate}
Then $g(\delta)$ begins in sector $S_{i-1}$ and  does not follow $\sigma_i$.
In particular, $g^i(\delta)$ begins in sector $S_0$ and does not follow $\sigma_1$.
\end{lem}


\begin{proof}
In case (a), since $T_\chi$ corresponds to a counterclockwise twist about $\chi$, we see that $T_\chi(\delta)$ initially follows $\sigma_i$  by Lemma~\ref{lem:blocker}(2); see Figure \ref{fig:Followingsigma_i}. In case (b), the arc $\delta$ follows $\sigma_i$, and by Lemma~\ref{lem:blocker}(2), the character $P(i)_u$ in $\delta$ blocks cancellation when applying $T_\chi$; see Figure \ref{fig:Followingsigma_i}. This implies that $T_\chi(\delta)$ follows $\sigma_i$ as well.  Applying $h$, we have that $g(\delta)$ begins in sector $S_{i-1}$.
Moreover, since $T_\chi(\delta)$ initially follows $\sigma_i$, there are two possibilities for the behavior of  $g(\delta)$: either $g(\delta)$ follows $\sigma_{i-1}$ when there are no punctures between $P(i)$ and $P(i-1)$, or $g(\delta)$ exits the sector $S_{i-1}$ immediately.  Either way, $g(\delta)$ does not follow $\sigma_{i}$.

\begin{figure}
\centering	
\begin{overpic}[width=3.25in]{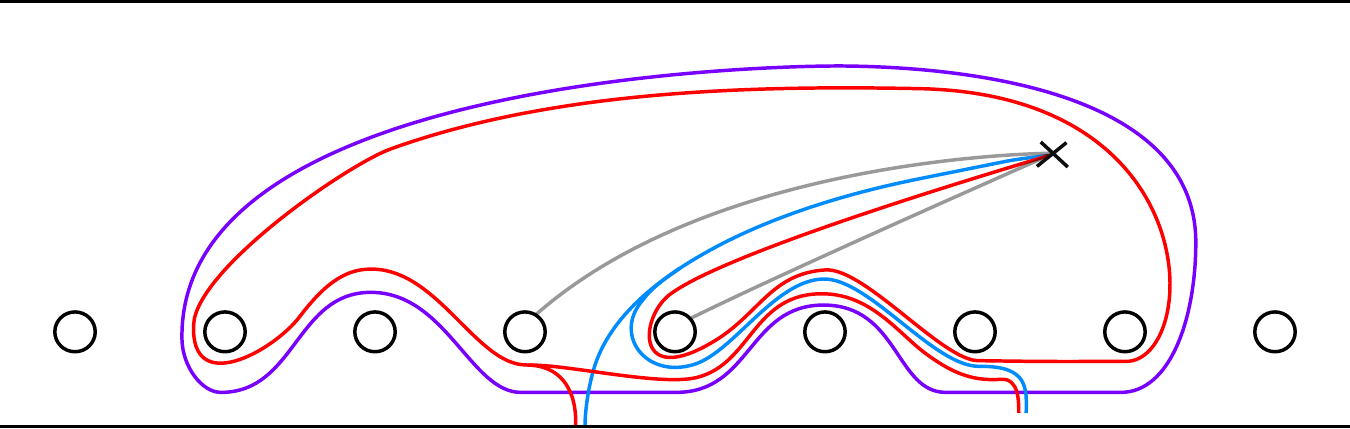}
\put(44,8){\textcolor{cyan}{\small$\delta$}}
\put(44,21){\textcolor{red}{\small$T_\chi(\delta)$}}
\end{overpic}
\caption{In both cases of Lemma~\ref{lemma:caseab}, $T_\chi(\delta)$ follows $\sigma_i$ so that $g(\delta)$ begins in $S_{i-1}$ and does not follow $\sigma_{i}$. Both examples of $\delta$ are shown in blue and $T_\chi(\delta)$ are shown in red.} 
\label{fig:Followingsigma_i}
\end{figure}

Applying this argument $i$ times yields the final statement of the lemma. 
\end{proof}


We now analyze the behavior of the image of $\delta$ when neither (a) nor (b) of Lemma \ref{lemma:caseab} holds, that is, when $\delta$ begins in sector $S_i$ and follows $\sigma_{i+1}$.

\begin{lem}\label{lemma:casec}
Let $\delta$ be an arc which begins in sector $S_i$ for some $i\in \{1,\dots,\ell-1\}$.
If $\delta$ initially follows $\sigma_{i+1}$, then there exists $s\in\{1,2\}$ such that $g^s(\delta)$ begins in sector $S_{i-1}$.
Moreover, if $g^s(\delta)$ follows $\sigma_i$, then $s=1$ and either
\begin{enumerate}[(i)]
\item $\delta$ follows $\sigma_{i+1}$ and then follows $\chi$ clockwise long enough to intersect the sector $S_i$ again; or 
\item $\delta$ hooks a puncture.
\end{enumerate}
\end{lem}


\begin{proof}
Since the sector $S_i$ is invariant under $T_\chi$ and applying $T_\chi$ cannot cause cancellation with the initial subarc of $\delta$ contained in $S_i$, it must be the case that $T_\chi(\delta)$  also begins in $S_i$.
Therefore, $T_\chi(\delta)$ must either initially follow $\sigma_i$, initially follow $\sigma_{i+1}$, or exit immediately.
In particular, since $i>0$, this implies that $g(\delta)=h(T_\chi(\delta))$ either begins in sector $S_{i-1}$ or exits $S_i$ immediately. 
In the latter case, $g(\delta)$ fits the hypothesis of Lemma \ref{lemma:caseab}(a) and therefore $g^2(\delta)$ is contained in sector $S_{i-1}$. 
This shows the first statement of the lemma.

We now prove the second statement of the lemma.
Suppose $g^s(\delta)$ follows $\sigma_i$.  If $s=2$, then we applied Lemma~\ref{lemma:caseab}(a) to conclude that $g^2(\delta)$ begins in sector $S_{i-1}$.  However, the moreover statement of that lemma shows that $g^2(\delta)$ does not follow $\sigma_i$.  Therefore $s=1$.

It remains to show that $\delta$ satisfies (i) or (ii).  Recall that we assume that $\delta$ follows $\sigma_{i+1}$.  There are three cases to consider depending on the behavior of $T_\chi(\delta)$.  In the second two cases, we determine the initial behavior of $\delta$ by taking the preimage of an initial segment of $T_\chi(\delta)$, i.e., the image of the segment under $T_\chi^{-1}$.  In general, taking preimages is only well-defined if the initial and terminal characters of the segment are fixed by the mapping class (see \cite[Section 3.4]{AMP} for details).  \\

\noindent {\bf Case 1:} If $T_\chi(\delta)$ follows $\sigma_i$, then $g(\delta)=h(T_\chi(\delta))$ doesn't follow $\sigma_i$.
Indeed, this was the case analyzed in the first paragraph of the proof of Lemma \ref{lemma:caseab} \\

\noindent {\bf Case 2:} Suppose $T_\chi(\delta)$ exits immediately. Since $g(\delta)$ follows $\sigma_i$, $T_\chi(\delta)$ must exit between $P(i)$ and $P(i)-1$.  The character immediately following where $T_\chi(\delta)$ exits is fixed by $T_\chi$ and blocks cancellation by Lemma~\ref{lem:blocker}.  Thus by taking the preimage of this initial segment under $T_\chi^{-1}$, we obtain an initial segment of $\delta$.  As seen in Figure \ref{fig:Lemma4.6}(A), (i) occurs.\\

\begin{figure}
\centering
\begin{subfigure}[t]{.5\textwidth}
\centering
\begin{overpic}[width=2.25in]{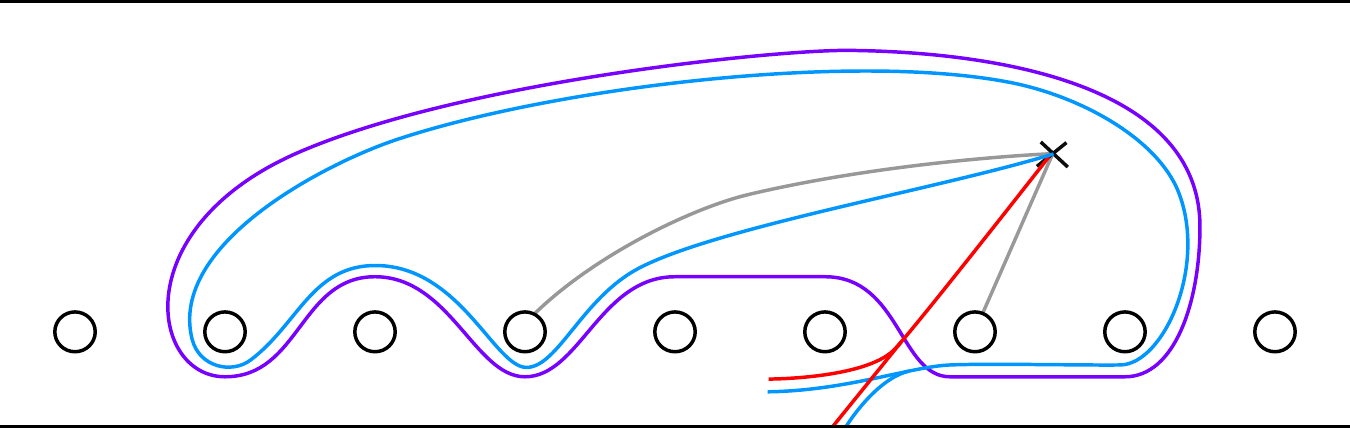}
\put(56,21.5){\textcolor{cyan}{\small $\delta$}}
\put(60.5,12.5){\textcolor{red}{\tiny $T_\chi(\delta)$}}
\end{overpic}
\caption{Case 2}
\end{subfigure}%
\begin{subfigure}[t]{.5\textwidth}
\centering
\begin{overpic}[width=2.25in]{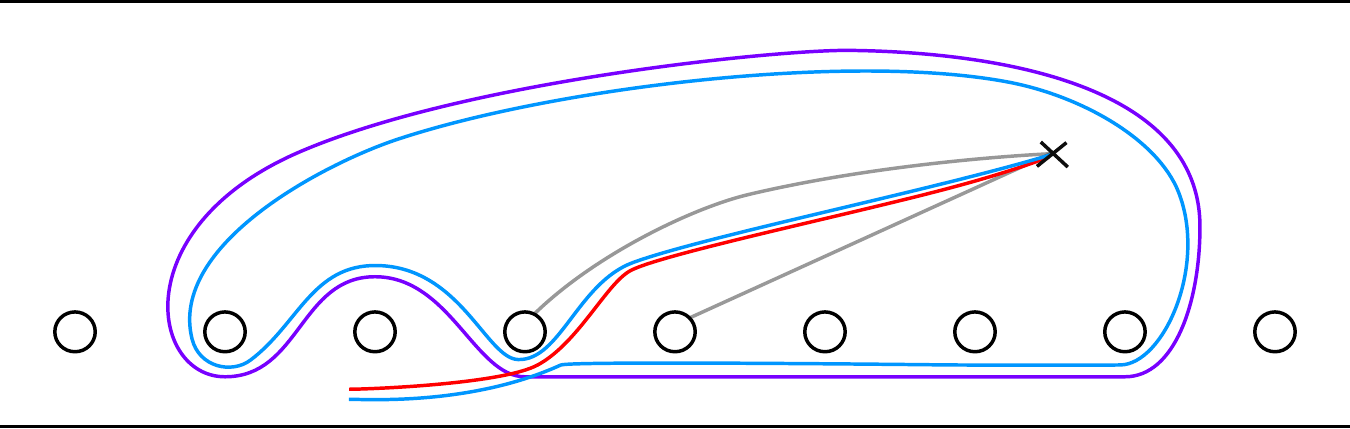}
\put(30,12.5){\textcolor{cyan}{\small $\delta$}}
\put(65,13.5){\textcolor{red}{\tiny $T_\chi(\delta)$}}
\end{overpic}
\caption{Case 3}
\end{subfigure}

\caption{Cases 2 and 3 of Lemma~\ref{lemma:casec}.  Note that the exterior punctures between $P(i+1)$ and $P(i)$ that appear in (A) do not appear in (B).  This is because, under the assumption of the lemma, $g(\delta)$ must follow $\sigma_i$.}
\label{fig:Lemma4.6}
\end{figure}

\noindent {\bf Case 3:} Suppose $T_\chi(\delta)$ follows $\sigma_{i+1}$.  Then after following $\sigma_{i+1}$, $T_\chi(\delta)$ follows $\chi$ in the clockwise direction for at least the character $P(i+1)_u$.  
Let $k_{o/u}$ be the first character that $T_\chi(\delta)$ stops following $\chi$.  Then $p_k$ is either an interior or exterior puncture. If $p_k$ is an exterior puncture, then $k_{o/u}$ is fixed by $T_\chi$ and blocks cancellation by Lemma~\ref{lem:blocker}.  By taking the image of this initial segment of $T_\chi(\delta)$ under $T_\chi^{-1}$, we see that (i) occurs.  On the other hand, if $p_k$ is an interior puncture, then $T_\chi(\delta)$ hooks a puncture, and so $k_{o/u}$ blocks cancellation, again by Lemma~\ref{lem:blocker}.  This initial segment is fixed by $T_\chi^{-1}$, and so $\delta$ also hooks a puncture, ensuring that (ii) holds.  See Figure \ref{fig:Lemma4.6}(B) for the first possibility.
\end{proof}


\begin{rem}
The reason that $S_\ell$ is excluded from the previous lemma is because $S_\ell$ is bounded by $\sigma_{\ell}$ and $\sigma_0$,
so that arcs in this sector  either follow $\sigma_\ell$ or exit immediately.
\end{rem}


The previous two lemmas show that arcs that begin in $S_i$ then begin in $S_{i-1}$ after applying either $g$ or $g^2$.  Iteratively applying these lemmas shows that, after applying $g$ a uniform number of times, the image of any arc $\delta$ begins in $S_0$.  Our goal is to show that the image of $\delta$ follows $\sigma_0$ under iterates of $g$; however this may not be the case the first time this image begins in sector $S_0$. The following lemma gives a precise description of when the image of $\delta$ follows $\sigma_1$.

\begin{cor}\label{cor:sectorshifter}
    Let $\delta$ be an arc which begins in sector $S_i$ for any $i\in\{1,\dots,\ell\}$.
There exists $0\leq k\leq 2\ell-1$ such that the following two properties hold:
\begin{itemize}
    \item the arc $g^{k}(\delta)$ begins in  sector $S_0$; and
    \item if the arc $g^k(\delta)$ follows $\sigma_1$, then $i=k$ and for all $0\le s\le k$, $g^s(\delta)$ begins in $S_{i-s}$, follows $\sigma_{i+1-s}$, and either
    \begin{enumerate}[(i)]
    \item follows $\chi$ clockwise long enough to return to sector $S_{i-s}$; or
    \item hooks a puncture.
    \end{enumerate}
\end{itemize}
  In the special case where the original arc $\delta$ hooks a puncture, then either $g^k(\delta)$ also hooks a puncture or $g^k(\delta)$ does not follow $\sigma_1$.
\end{cor}

\begin{proof}
This follows from an inductive application of Lemmas \ref{lemma:caseab} and \ref{lemma:casec},  keeping track of which of the different cases is occurring at each step.
\end{proof}

 
In the remainder of the section, we use the expanded notation $T_\chi^j$ instead of $T_\chi$ so we can easily record the power of the Dehn twist about $\chi$ that we are applying. 
This power did not affect the proofs before this point.
Let $\phi$ be the starts like function from  \eqref{eqn:startslike}.
Then the key result for showing that $(\alpha_i)_{i\in\N}$ is a quasi-geodesic axis is the following:

\begin{prop}\label{prop:startslikealpha1}
Suppose that either $\ell\ge 2$ or that $\ell=1$ and $g=hT_\chi^j$, where $j>1$.  If $\delta$ is an arc disjoint from $\mathring{\alpha}_0$, then there is some $0\leq k\leq 3\ell+1$ such that $g^{k}(\delta)$ starts like $\alpha_1$.
\end{prop}

Before beginning the proof, we prove one additional lemma which will be used frequently throughout the proof of the proposition.

\begin{lem}\label{lem:beginsinSj}
Under the assumptions of Proposition~\ref{prop:startslikealpha1}, if $\delta$ begins in sector $S_i$ for any $i\in\{0,\dots,\ell\}$ and exits immediately or follows $\sigma_i$, then $g^{k}(\delta)$ starts like $\alpha_1$ for some $0\leq k\leq i+\ell+1$.
\end{lem}

\begin{proof}
For $i \neq 0$, if $\delta$ exits immediately or follows $\sigma_i$, then Lemma \ref{lemma:caseab}(1) shows that $g^{i}(\delta)$ begins in sector $S_0$ and does not follow $\sigma_1$. Of course, when $i=0$ this fact about $\delta$ holds by assumption, so we allow for the possibility that $i=0$ as well.  If $g^i(\delta)$ follows $\sigma_0$, then it follows from Lemma \ref{lem:sigma0toalpha1} that $g^{i+1}(\delta)$ starts like $\alpha_1$.  

On the other hand, suppose $g^i(\delta)$ exits immediately.   Then $T^j_\chi(g^i(\delta))$ begins in $S_\ell$ and follows $\sigma_\ell$.  Thus $g^{i+1}(\delta)=h(T^j_\chi(g^i(\delta)))$ begins in $S_{\ell-1}$ and either follows $\sigma_{\ell-1}$ or exits immediately.  We now repeat the reasoning of this lemma for the element $g^{i+1}(\delta)$, and we obtain that $g^{i+\ell}(\delta)$ starts in $S_0$ and either follows $\sigma_0$ or exits immediately.  However, notice that since $P(1)=P(0)+1$, the only way an arc that begins in $S_0$ can exit immediately is if it goes to the back of the surface right after exiting.  By construction, $g^{i+\ell}(\delta)$ will not immediately go to the back; to see this, note that  applying $T^j_\chi$ to $\delta$ introduces a copy of $\chi$ before the arc goes to the back of the surface.  Since we are assuming that $\ell>1$ or $j\ge 2$, the image of some portion of this copy of $\chi$ will remain as we apply Lemma~\ref{lemma:caseab}, and it will always occur before the image of $\delta$ goes to the back. Therefore, $g^{i+\ell}(\delta)$ follows $\sigma_0$, and   $g^{i+\ell+1}(\delta)$ starts like $\alpha_1$ by Lemma~\ref{lem:sigma0toalpha1}.
 \end{proof}

We now turn to the proof of the proposition.
\begin{proof}[Proof of Propositin~\ref{prop:startslikealpha1}]
Suppose that $\delta$ begins in sector $S_i$ for any $i\in\{0,\dots,\ell\}$. Then by Corollary~\ref{cor:sectorshifter}, there exists $0 \leq k_0 \leq 2\ell -1$ such that $g^{k_0}(\delta)$ starts in sector $S_0$. If $g^{k_0}(\delta)$ starts like $\sigma_0$ or exits immediately,
Lemma~\ref{lem:beginsinSj} shows that $g^{{k_0}+\ell+1}(\delta)$ starts like $\alpha_1$. Let $j_0 = k_0 + \ell + 1 \leq 3\ell$.

By Corollary \ref{cor:sectorshifter}, the remaining case to analyze is when $k_0 = i$ and $g^i(\delta)$ begins in sector $S_0$ and follows $\sigma_1$.

If $i=0$, then $\delta$ itself begins in sector $S_0$ and follows $\sigma_1$.  In this case, since $\delta$ is disjoint from $\mathring{\alpha}_0$ by assumption, either $\delta$ hooks a puncture or it follows $\chi$ clockwise for at least the character $0_u$ but not long enough to intersect the sector $S_0$ again, and then exits.

 On the other hand, if $i\neq 0$, then we conclude from Corollary \ref{cor:sectorshifter} that $g^s(\delta)$ follows $\sigma_{i+1-s}$ for all $0\le s\le i$ and  that the initial segment of $g^s(\delta)$ satisfies either (i) or (ii) of that corollary.  
  Again using that $\delta$ is disjoint from $\mathring{\alpha}_0$, we see that $\delta$ cannot follow $\chi$ long enough to intersect the sector $S_i$ again. 
Thus, when $i\neq 0$, we must be in the case of Corollary \ref{cor:sectorshifter}(ii).  
  
To summarize, we are further reduced to the following two  cases:

\begin{enumerate}
\item $i\ge 0$ and $\delta$ begins in sector $S_i$, follows $\sigma_{i+1}$, and hooks a puncture.
\item $i=0$ and $\delta$ begins in sector $S_0$, follows $\sigma_{1}$, follows $\chi$ clockwise for at least one character but not long enough to intersect $S_0$ again, 
and then exits.  
Note that in this case $\delta$ does not hook a puncture.
\end{enumerate}

\noindent{\bf Case 1:} Assume first that (1) holds.   
Since $g^i(\delta)$ follows $\sigma_1$,  the moreover statement from Corollary~\ref{cor:sectorshifter} implies that $g^i(\delta)$ must also hook a puncture. This hooking blocks cancellation in $T^j_\chi(g^i(\delta))$ by Lemma~\ref{lem:blocker}.  Therefore, a reduced code for $T^j_\chi(g^i(\delta))$ begins with $P_s0_o0_u$, and so a reduced code for $g^{i+1}(\delta)$ begins with $P_s0_o1_o1_u$.  In particular,  $g^{i+1}(\delta)$ begins in sector $S_\ell$ and exits immediately.  By Lemma~\ref{lem:beginsinSj} applied to $g^{i+1}(\delta)$, there exists $0 \leq k_1\leq 2\ell + 1$ such that $g^{i+1+{k_1}}(\delta)$ starts like $\alpha_1$. Let $j_1 = i+1+ k_1 \le 3\ell+2$. \\

\noindent{\bf Case 2:} Next, assume that (2) holds, so that $i=0$  and $\delta$ begins in sector $S_0$, follows $\sigma_1$, then turns left and follows $\chi$ clockwise for at least one character but not long enough to return to sector $S_0$, and then exits. 
There are now two possibilities to consider, depending on whether $\delta$ exits in $S_\ell$ or not.  First, if $\delta$ exits in $S_r$ for some $0<r<\ell$, then $T^j_\chi(\delta)$ begins in $S_\ell$ and follows $\sigma_\ell$ (Figure~\ref{fig:WhereDeltaExits1}).  In particular, $g(\delta)$ begins in $S_{\ell-1}$ and does not follow $\sigma_\ell$.  Therefore, by applying Lemma~\ref{lem:beginsinSj} to $g(\delta)$, there is some $0\leq k_2\leq 2\ell$ such that $g^{1+{k_2}}(\delta)$ starts like $\alpha_1$. Let $j_2 = 1+ k_2 \leq 2\ell +1$.

\begin{figure}[h]
\begin{center}
\begin{subfigure}[t]{.5\textwidth}
\centering
\begin{overpic}[width=2.25in, trim={0 0 0 0}, clip]{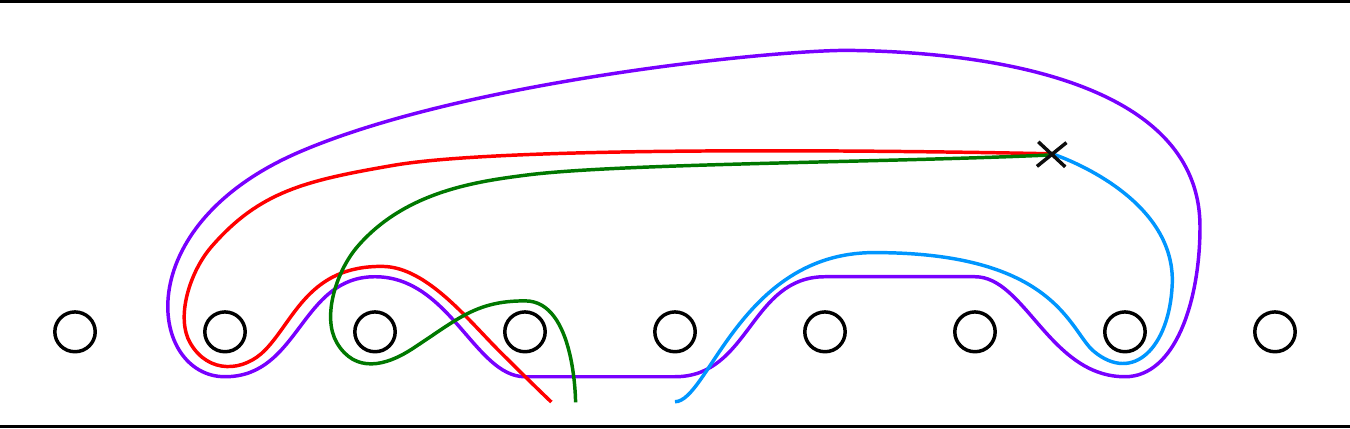}
\put(65,14){\textcolor{cyan}{\tiny$\delta$}}
\put(50,23){\textcolor{red}{\tiny$T_\chi(\delta)$}}
\put(30,14.5){\textcolor{ForestGreen}{\tiny$g(\delta)$}}
\end{overpic}
\vspace{.2in}
\end{subfigure}%
\begin{subfigure}[t]{.5\textwidth}
\begin{overpic}[width=2.25in, , trim={0 0 0 0}, clip]{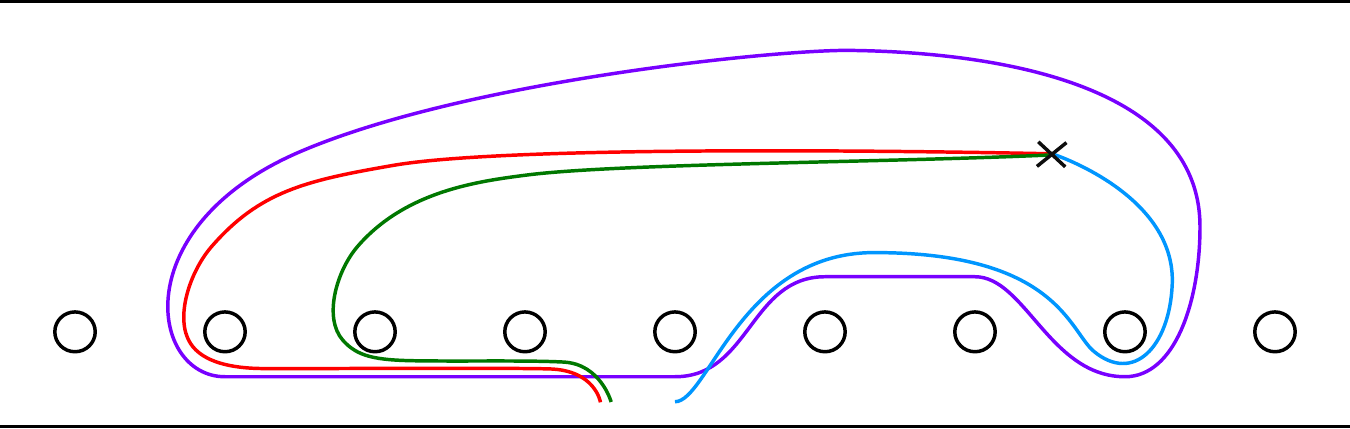}
\put(65,14){\textcolor{cyan}{\tiny$\delta$}}
\put(50,23){\textcolor{red}{\tiny$T_\chi(\delta)$}}
\put(30,14.5){\textcolor{ForestGreen}{\tiny$g(\delta)$}}
\end{overpic}
\end{subfigure}
\caption{Case (2) of the proof of Proposition~\ref{prop:startslikealpha1} if $\delta$ exits in $S_r$ for some $0 < r < l$.}
\label{fig:WhereDeltaExits1}
\end{center}
\end{figure}

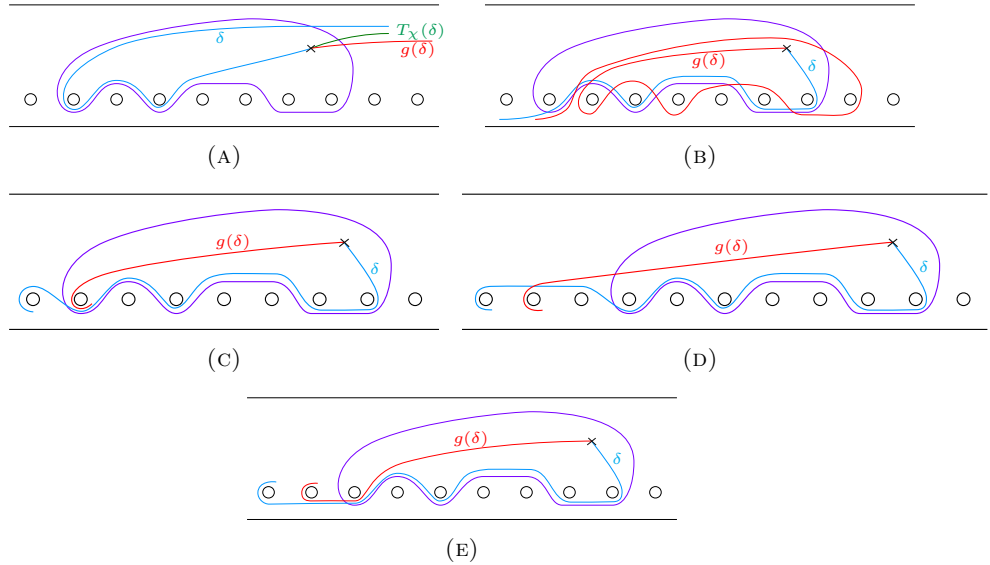
\begin{figure}[h]
\begin{center}
\begin{subfigure}[t]{.5\textwidth}
\centering
\begin{overpic}[width=2.25in]{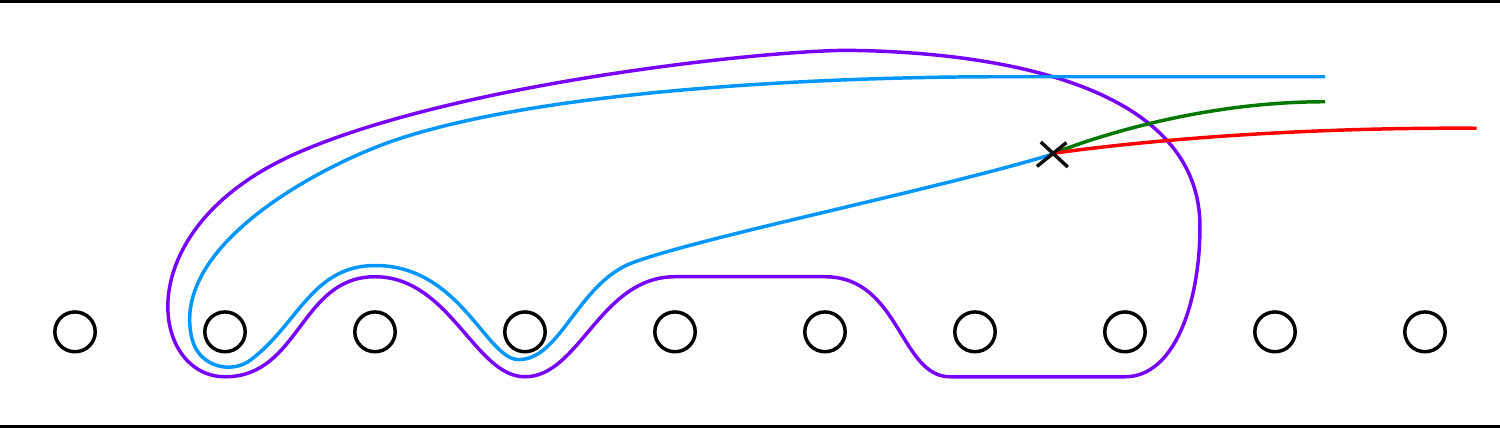}
\put(48,19.5){\textcolor{cyan}{{\tiny$\delta$}}}
\put(90,21.5){\textcolor{ForestGreen}{\tiny$T_\chi(\delta)$}}
\put(91,17){\textcolor{red}{\tiny$g(\delta)$}}
\end{overpic}
\caption{}
\end{subfigure}
\begin{subfigure}[t]{.5\textwidth}
\centering
\begin{overpic}[width=2.25in]{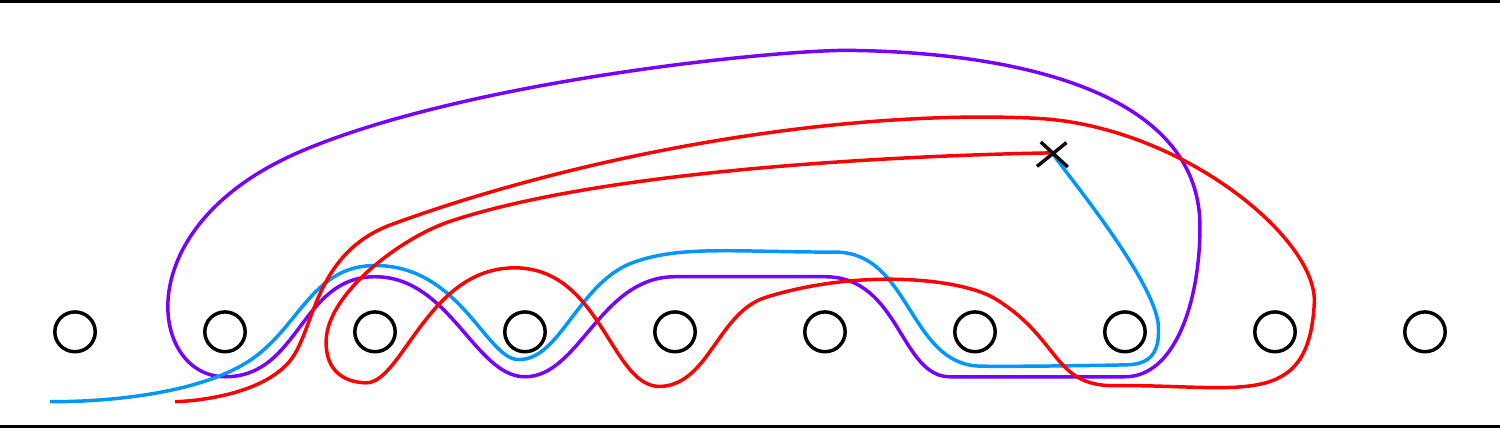}
\put(75,13){\textcolor{cyan}{\tiny$\delta$}}
\put(48,14.5){\textcolor{red}{\tiny$g(\delta)$}}
\end{overpic}
\caption{}
\end{subfigure}
\par\bigskip
\begin{subfigure}[t]{.5\textwidth}
\centering
\begin{overpic}[width=2.25in]{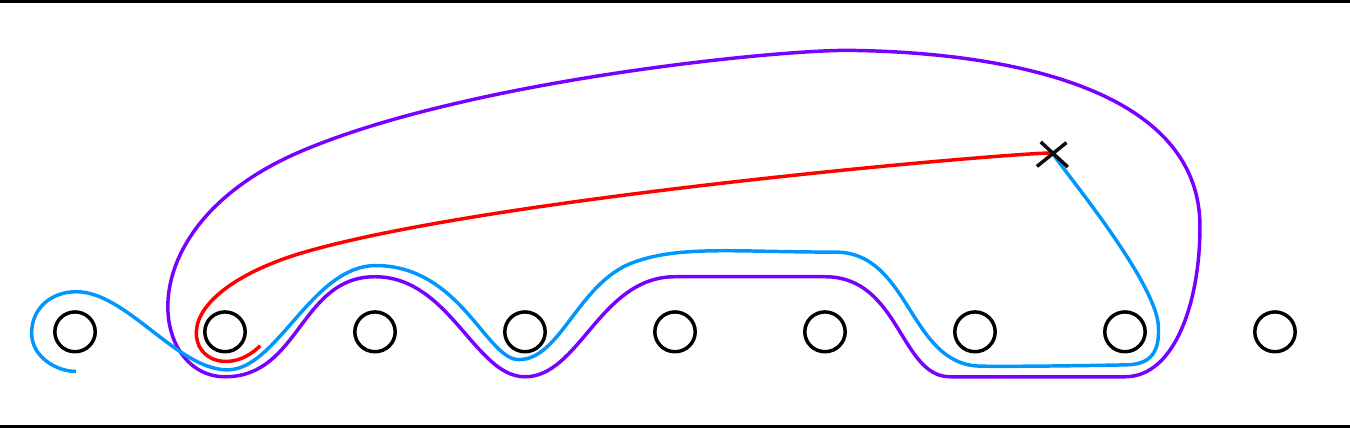}
\put(84,13){\textcolor{cyan}{\tiny$\delta$}}
\put(48,19.5){\textcolor{red}{\tiny$g(\delta)$}}
\end{overpic}
\caption{}
\end{subfigure}
\begin{subfigure}[t]{.5\textwidth}
\centering
\begin{overpic}[width=2.75in]{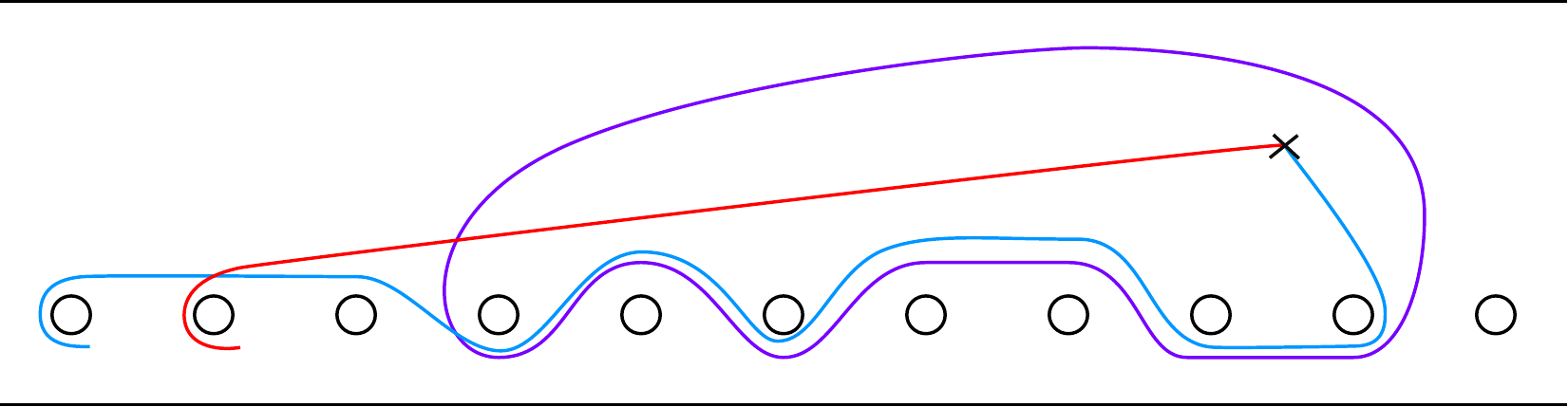}
\put(87,11){\textcolor{cyan}{\tiny$\delta$}}
\put(48,15){\textcolor{red}{\tiny$g(\delta)$}}
\end{overpic}
\caption{}
\end{subfigure}
\par\bigskip
\begin{subfigure}{.5\textwidth}
\centering
\begin{overpic}[width=2.25in]{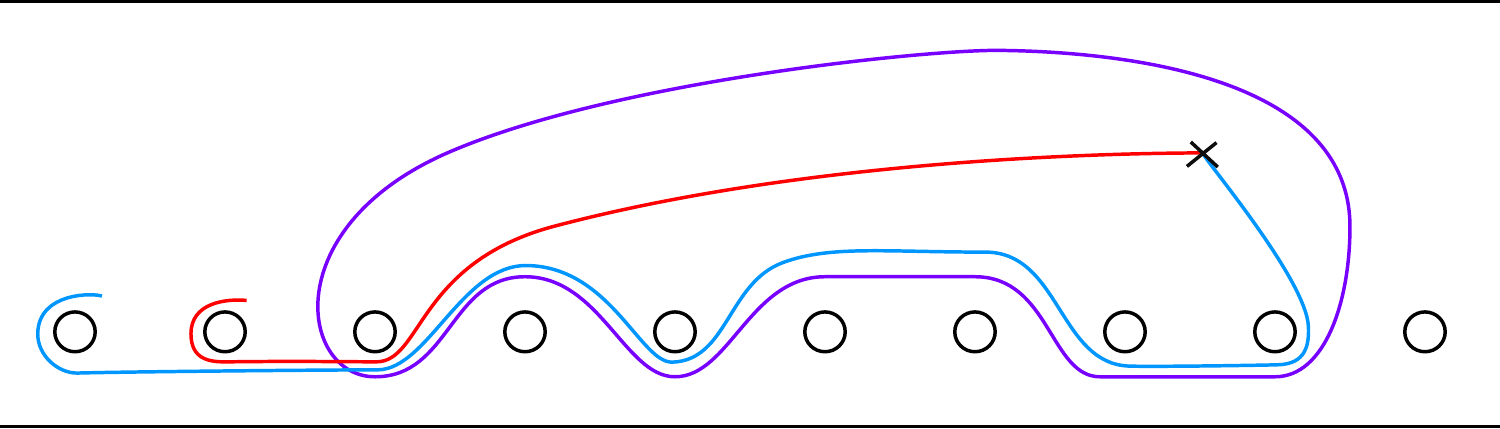}
\put(85,13){\textcolor{cyan}{\tiny$\delta$}}
\put(48,18){\textcolor{red}{\tiny$g(\delta)$}}
\end{overpic}
\caption{}
\end{subfigure}
\caption{Case (2) of the proof of Proposition~\ref{prop:startslikealpha1} if  $\delta$ begins in sector $S_0$, follows $\sigma_1$ and then exits in $S_\ell$.}
\label{fig:WhereDeltaExits}
\end{center}
\end{figure}

 On the other hand, suppose $\delta$ follows $\sigma_1$ and then exits in $S_\ell$. Let $q_{o/u}$ be the first character of a reduced code for $\delta$ after $\delta$ exits; we allow for the possibility that $q=C$.  If $q>P(\ell)$ or $q=C$ (Figure~\ref{fig:WhereDeltaExits}(A)), then $T^j_\chi(\delta)$ begins in $S_\ell$ and exits immediately. Therefore, so does $g(\delta)$. 
By Lemma~\ref{lem:beginsinSj} applied to $g(\delta)$, there exists $0 \leq k_3 \leq 2\ell+1$ such that $g^{k_3 +1}(\delta)$ starts like $\alpha_1$. Let $j_3 = k_3 + 1 \leq 2\ell+2$. If $q<P(\ell)$,  then we must have $q=P(\ell)-1$.  If $j\ge 2$, then $T^j_\chi(\delta)$  begins in $S_\ell$ and follows $\sigma_\ell$ (see Figure~\ref{fig:WhereDeltaExits}(B)). Therefore, $g(\delta)$ begins in $S_{\ell-1}$ and does not follow $\sigma_{\ell}$. The argument again concludes by applying Lemma~\ref{lem:beginsinSj}: there exists $0 \leq k_2 \leq 2\ell$ such that $g^{{k_2}+1}(\delta)$ starts like $\alpha_1$.

If $j=1$, then either: $g(\delta)$   begins in $S_{\ell-1}$ and follows $\sigma_\ell$; $g(\delta)$ begins in $S_\ell$ and follows $\sigma_\ell$; or $g(\delta)$ begins in $S_\ell$ and exits immediately.  Which of these occurs depends on the behavior of $\delta$ after exiting (see Figure \ref{fig:WhereDeltaExits}(C)-(E)).  In the latter two cases, Lemma~\ref{lem:beginsinSj} implies that there exists $0 \leq k_4 \leq 2\ell+1$ such that $g^{{k_4}+1}(\delta)$ starts like $\alpha_1$; let $j_4 = k_4 + 1 \leq 2\ell +2$. 
In the former case, note that $\ell\ge 2$ implies that $\ell-1\ge 1$, and hence $g(\delta)$ does not begin in sector $S_0$.
Moreover,  the initial form of $\delta$ ensures that $g(\delta)$ neither hooks a puncture nor follows $\chi$ long enough to intersect the sector $S_{\ell-1}$ again; see Figure~\ref{fig:WhereDeltaExits}(E).
In particular, applying Corollary \ref{cor:sectorshifter}, we conclude that $g^{k_0+1}(\delta)$ begins in sector $S_0$ and does not follow $\sigma_1$, for the constant $k_0 \leq 2\ell+1$ at the beginning of the proof. 
Applying Lemma \ref{lem:beginsinSj} as in the beginning of the proof, we conclude that there exists $j_{5}= k_0+ 1 + \ell + 1$ for which $g^{j_5}(\delta)$ starts like $\alpha_1$, where $j_5\le 3\ell+1$.

Taking $j = \max\{j_0, j_1, j_2, j_3, j_4, j_5\}$, we have that $j \leq 3\ell+1$ and that for any arc $\delta$ disjoint from $\mathring \alpha_0$, there is some $k \leq j$ such that $g^k(\delta)$ starts like $\alpha_1$.
\end{proof}


\begin{cor}\label{cor:BoundonStartsLike}
Suppose that either $\ell\ge 2$ or that $\ell=1$ and $g=hT_\chi^j$, where $j > 1$.
If $\phi(\gamma) = i$ for some $i>3\ell$ and $\delta$ is disjoint from $\gamma$, then $|\phi(\delta) - \phi(\gamma)| \leq 3\ell$.
\end{cor}

\begin{proof}
If $\gamma$ starts like $\alpha_i$ and $\delta$ is disjoint from $\gamma$, then $g^{-i}(\gamma)$ starts like $\alpha_0$ and $g^{-i}(\delta)$ is disjoint from $g^{-i}(\gamma)$ by Lemma~\ref{lem:shiftstartslike}. Thus, Proposition~\ref{prop:startslikealpha1} applied to $g^{-i}(\delta)$ shows that there exists some $k \leq 3\ell+1$ such that $g^{-i+k}(\delta)$ starts like $\alpha_1$. Applying Lemma~\ref{lem:shiftstartslike} again shows that $\delta$ starts like $\alpha_{i-k+1}$ and where $i-k+1 \geq i - 3\ell$, and so $\phi(\delta) \geq i-3\ell$.  Asymmetric argument swapping the roles of $\delta$ and $\gamma$ in the proof above shows that $\phi(\delta) \leq i + 3\ell$. Therefore,  if $\delta$ is disjoint from $\gamma$, then $|\phi(\delta) - \phi(\gamma)| \leq 3\ell$.
\end{proof}

\begin{rem}
If $g = hT_\chi$ (a single Dehn twist composed with $h$) and $\ell=1$, then Theorem \ref{thm:main} is actually false.
Indeed, $g$ will act elliptically with respect to the action on $\mc A(S,p)$. To see this, consider the infinite sequence of pairwise disjoint arcs $(\delta_i)_{i\in\mathbb{Z}}$ for which $\delta_i=g^i(\delta_0)$, as in Figure~\ref{fig:Ellipticg}.
Such arcs form an infinite $1$-simplex in $\mc A(S,p)$ preserved by $g$.

\begin{figure}
    \centering
    \begin{overpic}[width=3in]{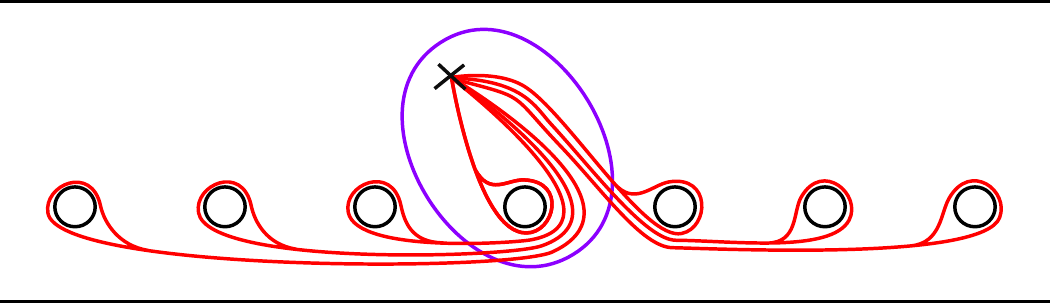}
    \put(5,12.5){\small$\delta_{-3}$}
    \put(19.5,12.5){\small$\delta_{-2}$}
    \put(33.5,12.5){\small$\delta_{-1}$}
    \put(47.5,12.5){\small$\delta_{0}$}
    \put(63.5,12.5){\small$\delta_{1}$}
    \put(77.5,12.5){\small$\delta_{2}$}
    \put(91.5,12.5){\small$\delta_{3}$}
    \end{overpic}
    \caption{A collection of pairwise disjoint arcs $(\delta_i)_{i\in\mathbb{Z}}$, represented as train tracks, such that $\cup_i\delta_i$ is fixed by the composition of a shift and single twist about $\chi$, when $\ell=1$.}
    \label{fig:Ellipticg}
\end{figure}
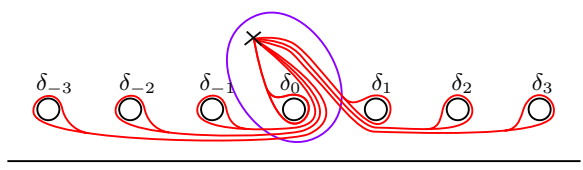

\end{rem}


\section{Proof of main theorem}

Let $\phi$ be the starts like function defined in \eqref{eqn:startslike}, and let $M=3\ell$.

\begin{lem}\label{lem:phitoASp}
For any $\gamma,\delta\in\mc A(S,p)$, we have $d_{\mc A(S,p)}(\gamma,\delta)\geq \frac{1}{M}\vert \phi(\gamma)-\phi(\delta)\vert$.
\end{lem}

\begin{proof}
Suppose $d_{\mc A(S,p)}(\gamma,\delta)=1$, and assume without loss of generality that $\phi(\delta)\leq \phi(\gamma)$.  If $\phi(\gamma)=j$, then $g^k(\gamma)$ begins like $g^k(\alpha_j)=\alpha_{j+k}$ for any $k\geq 0$, by Lemma~\ref{lem:shiftstartslike}.  Since $d_{\mc A(S,p)}(-,-)$ is $\Map(S,p)$--invariant, by replacing $\gamma$ and $\delta$ with $g^k(\gamma)$ and $g^k(\delta)$, respectively, we may assume that $j\geq M$.   By Corollary~\ref{cor:BoundonStartsLike}, it follows that $\vert \phi(\delta)-\phi(\gamma)\vert\leq M$.  We conclude using the subadditivity of the absolute value.
\end{proof}

\begin{proof}[Proof of Theorem~\ref{thm:main}]
It suffices to show that $\left(\alpha_j\right)_{j\geq M}$ is a quasi-geodesic half-axis for $g$.  To see this, note that this is an unbounded orbit of $g$, so $g$ is not elliptic, and since $g$ acts as translation along this half-axis, it cannot be parabolic. 

Let $f\colon\Z_{\ge 0}\to \mc A(S,p)$ be the map $i\mapsto \alpha_{M+i}=g^{M+i}(\alpha_0)$.  Since $\phi(\alpha_j)=j$ for any $j\geq 0$, it follows from Lemma~\ref{lem:phitoASp} that for all $i\geq 0$, \[d_{\mc A(S,p)}(\alpha_{M},\alpha_{M+i})\geq \frac{i}{M}. \]
Moreover, the arc $\beta=P_s0_o1_o1_u0_oP_s$ is disjoint from $\alpha_0$ and $\alpha_1$, and so $g^j(\beta)$ is disjoint from $\alpha_j$ and $\alpha_{j+1}$ for all $j\in\Z$.  Thus $d_{\mc A(S,p)}(\alpha_j,\alpha_{j+1})\leq 2$, and so for all $i\geq 0$,
\[d_{\mc A(S,p)}(\alpha_{M},\alpha_{M+i})\leq 2i.\]  Therefore, the map $f$ is a $(2M,0)$--quasi-isometric embedding.
\end{proof}

\vspace{.25 in}

\bibliographystyle{abbrv}
\bibliography{Biblio}

\begin{thebibliography}{1}

\bibitem{AMP}
C.~Abbott, N.~Miller, and P.~Patel.
\newblock Infinite-type loxodromic isometries of the relative arc graph.
\newblock Preprint: arXiv:2109.06106.

\bibitem{Bavard}
J.~Bavard.
\newblock Hyperbolicit\'{e} du graphe des rayons et quasi-morphismes sur un
  gros groupe modulaire.
\newblock {\em Geom. Topol.}, 20(1):491--535, 2016.

\bibitem{BavardWalker}
J.~Bavard and A.~Walker.
\newblock The {G}romov boundary of the ray graph.
\newblock {\em Trans. Amer. Math. Soc.}, 370(11):7647--7678, 2018.

\bibitem{MoralesValdez}
I.~Morales and F.~Valdez.
\newblock Loxodromic elements in big mapping class groups via the
  {H}ooper--{T}hurston--{V}eech construction.
\newblock Preprint: arXiv:2003.00102.

\bibitem{PatelTaylor}
P.~Patel and S.~Taylor.
\newblock Constructing endperiodic loxodromics of infinite-type arc graphs.
\newblock Preprint: arXiv:2211.00678.

\end{thebibliography}

\end{document}